\documentclass[11pt]{article}
\usepackage[T1]{fontenc}
\usepackage[utf8]{inputenc}
\usepackage{lmodern,microtype}
\usepackage[a4paper,margin=26mm,headheight=15pt,footskip=13mm]{geometry}
\usepackage{amsmath,amssymb,amsthm,mathtools,mathrsfs}
\usepackage{enumitem,booktabs,tabularx,longtable,array}
\usepackage{fancyhdr,needspace}
\usepackage{xurl}
\usepackage[hidelinks,bookmarksnumbered=true]{hyperref}
\setlist[enumerate]{leftmargin=2em,itemsep=3pt,topsep=5pt}
\allowdisplaybreaks[2]
\numberwithin{equation}{section}
\newtheorem{theorem}{Theorem}[section]
\newtheorem{lemma}[theorem]{Lemma}
\newtheorem{conjecture}[theorem]{Conjecture}
\theoremstyle{definition}

\hypersetup{pdftitle={Symplectic capacities and Weyl laws from Abelian–Higgs min–max theory},pdfauthor={Taisuke Shibata}}
\newcommand{\eps}{\varepsilon}
\newcommand{\R}{\mathbb R}
\newcommand{\C}{\mathbb C}
\newcommand{\Z}{\mathbb Z}
\newcommand{\CP}{\mathbb{CP}}
\newcommand{\G}{\mathcal G}
\newcommand{\Conf}{\mathcal C}
\newcommand{\B}{\mathcal B}
\newcommand{\LL}{\mathscr L}
\newcommand{\En}{\mathcal E}
\newcommand{\Et}{\widetilde{\mathcal E}}
\newcommand{\AH}{\mathrm{AH}}
\newcommand{\cAH}{c^{\AH}}
\newcommand{\Vol}{\operatorname{Vol}}
\newcommand{\Area}{\operatorname{Area}}
\newcommand{\Mass}{\mathbf M}
\newcommand{\av}{\operatorname{av}}
\newcommand{\Id}{\operatorname{id}}
\newcommand{\spt}{\operatorname{spt}}
\newcommand{\supp}{\operatorname{supp}}
\newcommand{\Sym}{\operatorname{Sym}}
\newcommand{\diver}{\operatorname{div}}

\newcommand{\PD}{\operatorname{PD}}

\newcommand{\arxiv}[1]{\href{https://arxiv.org/abs/#1}{\nolinkurl{arXiv:#1}}}

\theoremstyle{definition}
\newtheorem*{question}{Question}

\title{Symplectic capacities and Weyl laws from Abelian–Higgs min–max theory}
\author{Taisuke Shibata}
\date{1 October 2026\\\small}
\begin{document}
\maketitle
\begin{abstract}

For a Hermitian line bundle $E\to Y$ over a Riemannian manifold
$(Y^n,g)$, the self-dual Abelian-Higgs energy is defined by
$\mathcal{E}^Y_{\varepsilon,g}(u,a)=\frac{1}{2\pi}\int_Y
\left(|\nabla_a u|_g^2+\varepsilon^2|F_a|_g^2+
\frac{(1-|u|^2)^2}{4\varepsilon^2}\right)dV_g$,
where $u$ is a smooth section of $E$, $a$ is a unitary
connection, and $F_a$ is its curvature. Pigati and Stern studied
the energy as a codimension-two analogue of the Allen-Cahn
energy to study minimal submanifolds in a Riemannian manifold.

In this paper, we use this functional to define a
nondecreasing sequence of symplectic capacities
$\{c_k^{\mathrm{AH}}\}_{k\ge1}$ in dimension four by a min-max
construction. In particular, we prove that
$c_1^{\mathrm{AH}}$ is normalized on the ball and the cylinder.
For the upper bound on the cylinder, we combine
Gromov's theory of families of $J$-holomorphic spheres with
Taubes's planar vortex. For the lower bound on the ball,
we use the properties of the Abelian-Higgs theory developed by
Pigati-Stern, together with a geometric estimate on smooth
approximations to the double of the ball and an equivariant
deformation of low energy sublevels. These estimates give
the normalization without Seiberg-Witten theory.

We then use Seiberg-Witten moduli spaces on $\mathbb{CP}^2$ and
$S^2\times S^2$ to prove an upper bound on the higher
capacities of balls and the formula
$c_k^{\mathrm{AH}}(Z^4(a),\omega_0)=ka$.
We also prove Weyl laws for domains in $\mathbb{R}^4$, both for
the capacity sequence and for the min-max sequence
with the standard Euclidean metric. For these results,
we construct AH sweepouts from moduli spaces of the solutions
to the perturbed Seiberg-Witten equations.
\end{abstract}

\tableofcontents

\Needspace{15\baselineskip}
\section{Introduction}\label{sec:introduction}

In this section, we first give the definitions and state the main theorems. We then
explain the historical background and the motivation of this paper.
Finally, we give an outline of the proofs.

\subsection{The AH energy and capacity sequence}\label{sec:bundle-AH}\label{sec:four-dimensional-definitions}
Let $(Y^n,g)$ be a smooth Riemannian manifold, and let $E\to Y$
be a Hermitian line bundle. We consider a smooth section $u$ of $E$
and a unitary connection $a$. We write $F_a=\nabla_a^2\in i\Omega^2(Y)$
for the curvature of this connection. The Abelian--Higgs energy of
this pair is defined by

\begin{equation}\label{eq:unitary-AH-energy}
 \En^Y_{\eps,g}(u,a)=\frac1{2\pi}\int_Y
 \left(|\nabla_a u|_g^2+\eps^2|F_a|_g^2+
       \frac{(1-|u|^2)^2}{4\eps^2}\right)dV_g.
\end{equation}

This functional is the self-dual $U(1)$ Yang--Mills--Higgs functional
studied by Pigati--Stern \cite{PigatiStern2021}. In our convention,
we divide their energy by $2\pi$. When we use the energy without
this normalization, we write $E_{\eps,g}=2\pi\En_{\eps,g}$.
In a local unitary frame, we can write the connection and its
curvature in the following form:

\begin{equation}\label{eq:SW-AH-notation}
 \nabla_a=d-iA_a,\qquad F_a=-i\,dA_a.
\end{equation}

Unitary gauge transformations do not change the value of this energy.

We now consider the four-dimensional case. Let
$D=\operatorname{Int}\overline D$, where $\overline D$ is a compact
connected oriented smooth four-manifold. We allow its boundary to
be empty. The metric, fields, and gauge transformations are required
to be smooth up to the boundary. We also use Euclidean cubes as
auxiliary domains. In this case, we require smooth extensions to a
neighborhood of the closure. We use the trivial bundle $E_D=D\times\C$
and write $d_A=d-iA$. In this notation, the energy is given by

\begin{equation}\label{c:eq:energy}
 \En^D_{\eps,g}(u,A)=\frac{1}{2\pi}\int_D
 \left(|d_Au|_g^2+\eps^2|dA|_g^2+
 \frac{(1-|u|^2)^2}{4\eps^2}\right)dV_g.
\end{equation}

We define the gauge group and the configuration quotient as follows:

\begin{equation}\label{c:eq:quotient}
 \G(D)=C^\infty(\overline D,S^1),\qquad
 \B^*(D)=\Conf^*(D)/\G(D),\qquad
 \Conf^*(D)=\{(u,A):u\not\equiv0\}.
\end{equation}

Here the fields satisfy $u\in C^\infty(\overline D;\C)$ and
$A\in\Omega^1(\overline D;\R)$. The gauge action on a pair is given by
$\gamma\cdot(u,A)=(\gamma u,A-i\gamma^{-1}d\gamma)$.

We give $\B^*(D)$ the quotient topology coming from the smooth
Fr\'echet topology. For a point $x\in D$, we evaluate each gauge
transformation at $x$. This evaluation gives the following
line bundle over the quotient:

\begin{equation}\label{c:eq:eval-line}
 \LL_x=(\Conf^*(D)\times\C)/\G(D),\qquad
 \gamma\cdot(c,z)=(\gamma\cdot c,\gamma(x)z).
\end{equation}

We denote the first Chern class of this line bundle by
$\mu_D=c_1(\LL_x)\in H^2(\B^*(D);\Z)$. This class does not depend
on the choice of $x$.

As in the usual construction, a finite-dimensional space of smooth
sections gives a projective family. The pullback of the evaluation
line to this family is $\mathcal O(1)$. In particular, $\LL_x$ is
nontrivial and $\mu_D^k\ne0$ for every $k\ge1$. We use such projective
families in the proof of Lemma~\ref{c:lem:elementary}.

\Needspace{6\baselineskip}
We use this evaluation class to give the topological condition for
our min--max construction. We call a pair $(K,F)$ an \emph{AH $k$-sweepout}
if $K$ is a nonempty finite simplicial complex and
$F:K\to\B^*(D)$ is a continuous map which satisfies

\begin{equation}\label{eq:family-detection}
 F^*(\mu_D^k)=(F^*\mu_D)^k\ne0
 \quad\text{in }H^{2k}(K;\Z).
\end{equation}

We write $\mathscr S_k(D)$ for the collection of these sweepouts.
We define

\begin{align}
 w_{\eps,k}(D,g)
 &=\inf_{(K,F)\in\mathscr S_k(D)}
       \max_{z\in K}\En^D_{\eps,g}(F(z)),
       \label{eq:family-minmax}\\
 W_k(D,g)&=\liminf_{\eps\downarrow0}w_{\eps,k}(D,g).
       \label{eq:family-scale-limit}
\end{align}

For a connected smooth symplectic four-manifold $(M,\omega)$, define

\begin{equation}\label{eq:family-capacity}
 \cAH_k(M,\omega)
 =\sup_{D\Subset\operatorname{Int}M}
   \sup_{J\in\mathcal J_\omega(\overline D)}W_k(D,g_J),
 \qquad g_J(v,w)=\omega(v,Jw).
\end{equation}

For the first supremum, we take connected domains $D$ whose closures
are compact and smooth and are contained in $\operatorname{Int}M$.
For the second supremum, we take smooth compatible almost complex
structures defined on a neighborhood of $\overline D$. If $M$ is
closed, we also allow the choice $D=M$. We set $\cAH_0=0$.

In particular, we can use this definition on every connected open
symplectic four-manifold. This includes a connected open set
$U\subset\R^4$.

\subsection{Axioms  of symplectic capacities and estimates on balls and cylinders}\label{sec:capacity-axioms}
Before stating our first result, we specify the standard models
and the capacity axioms which we use. We write $z_j=x_j+iy_j$ and set

\begin{equation}\label{eq:standard-symplectic-form}
 \omega_0=\sum_{j}\frac{i}{2} dz_{j}\wedge d\bar{z_{j}},\qquad
 g_0=\sum_{j} (dx_j^2+dy_j^2).
\end{equation}

The open ball and the open cylinder are denoted by
\begin{equation}\label{eq:standard-models}
 B^{2n}(a)=\{z\in\C^n:\pi|z|^2<a\},\qquad 
 Z^{2n}(a)=B^{2}(a) \times \C^{n-1}.
\end{equation}

On both sets, we use the restriction of $\omega_0$. We sometimes just  write $B(a)$, $ 
 Z(a)$ if there is no confusion.  

We use the capacity axioms in \cite[Chapter 2]{HoferZehnder1994}.
 In this paper, a symplectic capacity
on symplectic $2n$-manifolds assigns a value
$c(M^{2n},\omega)\in[0,+\infty]$ and has the following properties:
\begin{enumerate}[label=\textup{(C\arabic*)}]
\item \emph{Monotonicity.} A symplectic embedding
$\varphi:(M^{2n},\omega)\hookrightarrow(N^{2n},\eta)$ implies
$c(M^{2n},\omega)\le c(N^{2n},\eta)$.
\item \emph{Conformality.} For $r>0$,
$c(M^{2n},r\omega)=r\,c(M^{2n},\omega)$.
\item \emph{Nontriviality.}
$0<c(B^{2n}(1),\omega_0)$ and $c(Z^{2n}(1),\omega_0)<+\infty$.
\end{enumerate}
We call a capacity $c$ \emph{normalized} if
$c(B^{2n}(1),\omega_0)=c(Z^{2n}(1),\omega_0)=1$.
With this convention, we state our first main result as follows.

\Needspace{14\baselineskip}

\begin{theorem}\label{c:thm:main}
For each $k\ge1$, the invariant $\cAH_k$ is a symplectic capacity on
connected symplectic four-manifolds. The sequence is nondecreasing in
$k$. For every $a>0$, we have
\begin{equation}\label{c:eq:main-normalization}
 \cAH_1(B^4(a),\omega_0)=\cAH_1(Z^4(a),\omega_0)=a.
\end{equation}
Moreover,
\begin{equation}\label{c:eq:raw-first}
 \lim_{\eps\downarrow0}w_{\eps,1}(B^4(a),g_0)=a.
\end{equation}
\end{theorem}

We prove this theorem in Section~\ref{c:sec:capacity} without
Seiberg--Witten theory. We then use Seiberg--Witten theory
to obtain estimates for the higher capacities of balls
and cylinders. These estimates are proved in
Section~\ref{m:sec-upper}.

\Needspace{10\baselineskip}
\begin{theorem}\label{m:thm:models}
For every $a>0$ and $k\ge1$, we have
\begin{equation}\label{eq:intro-cylinder-exact}
 \cAH_k(Z^4(a),\omega_0)=ka
\end{equation}
and
\begin{equation}\label{eq:intro-ball-upper}
 a\le\cAH_k(B^4(a),\omega_0)\le a\,d(k),\qquad
 d(k)=\min\{d\in\Z_{\ge1}:d(d+3)/2\ge k\}.
\end{equation}
\end{theorem}

\subsection{Weyl laws for domains in \texorpdfstring{$\R^4$}{R4}}
\label{sec:intro-weyl}\label{sec:intro-standard-metric}
By \eqref{eq:family-capacity}, the capacity sequence is defined for
every nonempty connected open set $U\subset\R^4$. We write
$|U|=\frac12\int_U\omega_0\wedge\omega_0$ for the four-dimensional
Lebesgue measure. In the following theorem, $\partial U$ and
$\overline U$ mean the boundary and the closure taken in  $\R^4$.

\Needspace{10\baselineskip}

\begin{theorem}\label{thm:intro-weyl}
There exists a universal constant $b_{\AH}$ with
\begin{equation}\label{eq:intro-coefficient}
 \sqrt2\le b_{\AH}\le2
\end{equation}
such that the following statements hold for every nonempty connected
open set $U\subset\R^4$. If $U$ is bounded and $|\partial U|=0$, then
\begin{equation}\label{eq:intro-optimized-weyl}
 \lim_{k\to\infty}\frac{\cAH_k(U,\omega_0)}{\sqrt{k}}
 =b_{\AH}\sqrt{|U|}.
\end{equation}
 For every nonempty connected
open set $U$, without any boundedness or boundary assumption,
\begin{equation}\label{eq:intro-general-lower}
 \liminf_{k\to\infty}\frac{\cAH_k(U,\omega_0)}{\sqrt{k}}
 \ge b_{\AH}\sqrt{|U|}.
\end{equation}
If $U$ is bounded, we also have
\begin{equation}\label{eq:intro-general-upper}
 \limsup_{k\to\infty}\frac{\cAH_k(U,\omega_0)}{\sqrt{k}}
 \le b_{\AH}\sqrt{|\overline U|}.
\end{equation}

\end{theorem}

We also state a Weyl law before taking the supremum over compatible
metrics. We use $W_k(D,g_0)$ as defined in
\eqref{eq:family-scale-limit}. In particular, the standard Euclidean
metric is fixed.

\Needspace{12\baselineskip}
\begin{theorem}\label{thm:intro-fixed-weyl}
There exists a universal constant $b_{\AH}^{\mathrm{flat}}$ with
\[
 0<b_{\AH}^{\mathrm{flat}}\le b_{\AH}\le2
\]
such that, for every nonempty bounded connected smooth domain
$D\subset\R^4$,
\begin{equation}\label{eq:intro-fixed-weyl}
 \lim_{k\to\infty}\frac{W_k(D,g_0)}{\sqrt{k}}
 =b_{\AH}^{\mathrm{flat}}\sqrt{|D|}.
\end{equation}

\end{theorem}

\subsection{Historical background and motivation}
\label{sec:intro-background}

The main motivation of this paper is to understand the relation
between Reeb dynamics and the min--max theory of minimal
submanifolds. We study this relation through the Abelian--Higgs
energy and Seiberg--Witten theory. We first recall the two
settings and their Weyl laws.

In Riemannian geometry, Gromov introduced a sequence of min--max
invariants by using sweepouts and cup powers of cohomology classes
\cite{Gromov2003}. Almgren studied the topology of cycle spaces
\cite{Almgren1962}. For hypersurfaces, we use the component
containing zero in the space of $\Z/2\Z$ cycles and the generator
of its cohomology with $\Z/2\Z$ coefficients. A $p$-sweepout
detects the $p$th cup power of this generator. Min--max over these
sweepouts gives the $p$-width $\omega_p(Y,g)$. The sequence is
nondecreasing and is called the \emph{volume spectrum}.
Gromov and Guth proved bounds for the growth of these values
\cite{Gromov2003,Guth2009}. Marques--Neves used the volume spectrum
in the Almgren--Pitts min--max theory of minimal hypersurfaces
\cite{Pitts1981,MarquesNeves2017}.

Gromov conjectured a Weyl law for this spectrum
\cite{Gromov2003}. Liokumovich--Marques--Neves proved this
conjecture \cite[Theorem~1.1]{LMN2018}. For a closed connected
Riemannian manifold $(Y^n,g)$, they proved
\[
 \lim_{p\to\infty}\frac{\omega_p(Y,g)}{p^{1/n}}
 =a_n\Vol_g(Y)^{(n-1)/n},\qquad a_n>0.
\]
Here $a_n$ depends only on the dimension. They also proved a
Weyl law for cycles of higher codimension on Euclidean domains
under a homological assumption \cite[Theorem~1.3]{LMN2018}.
Guth--Liokumovich proved a Weyl law for one-dimensional cycles
with $\Z/2\Z$ coefficients on compact Riemannian three-manifolds
\cite[Theorem~1.1]{GuthLiokumovich2025}. This is a Weyl law in
codimension two. Staffa extended the result for one-dimensional
cycles to compact Riemannian manifolds of arbitrary dimension
\cite[Theorem~1.1]{Staffa2024}. These results use $\Z/2\Z$
coefficients. To our knowledge, the analogous Weyl law for
integral cycles in codimension two remains open in general.

In symplectic geometry, capacities were introduced to measure
the size of symplectic manifolds and to give obstructions to
symplectic embeddings
\cite{EkelandHofer1989,HoferZehnder1994}.
Ekeland--Hofer used equivariant min--max for the Hamiltonian
action functional to construct an infinite sequence of
capacities \cite{EkelandHofer1990}. We now recall some
constructions in dimensions three and four.

Taubes proved that the Seiberg--Witten invariants of closed
symplectic four-manifolds agree with the Gromov invariants
defined by pseudoholomorphic curves \cite{Taubes2000}.
In dimension three, he proved that Hutchings's embedded
contact homology (ECH) is isomorphic to monopole Floer
cohomology \cite[Theorem~1]{Taubes2010ECHI}
($\mathrm{ECH}\cong\mathrm{HM}$).
Hutchings defined ECH capacities by using filtered ECH on the
boundary of a four-dimensional Liouville domain
\cite[Section~4.1]{Hutchings2011}. This gives an infinite sequence
of symplectic capacities. The proof of monotonicity under
symplectic embeddings uses cobordism maps from Seiberg--Witten
theory.

Cristofaro-Gardiner--Hutchings--Ramos proved a Weyl law for ECH
capacities \cite{CGHR2015}. Let $(X,\omega)$ be a four-dimensional
Liouville domain, and suppose that all of its ECH capacities
are finite. Then
\[
 \lim_{k\to\infty}\frac{c_k^{\mathrm{ECH}}(X,\omega)^2}{k}
 =4\Vol_\omega(X),\qquad
 \Vol_\omega(X)=\frac12\int_X\omega\wedge\omega.
\]
This formula follows from their Weyl law for ECH spectral
invariants on closed contact three-manifolds
\cite{CGHR2015}.

Hutchings later defined another sequence of capacities using
pseudoholomorphic curves \cite{Hutchings2022}. For a fixed
almost complex structure and $k$ distinct points, he takes
the infimum of the energies of curves through these points.
He then takes the supremum over the almost complex structures
and the points. For a Liouville domain, the curves lie in its
completion, and the definition extends to symplectic
four-manifolds. The definition and the proofs of the capacity
axioms do not use ECH or Seiberg--Witten theory.

These alternative capacities agree with ECH capacities on
convex and concave toric domains and satisfy a Weyl law
for compact smooth domains in $\R^4$
\cite{Hutchings2022}. Beiner studied these capacities on
closed symplectic four-manifolds and proved a Weyl law with
bounded error when $b_2^+=1$ \cite{Beiner2026}.
In the algebraic setting, Wormleighton and
Chaidez--Wormleighton studied algebraic capacities of
polarized algebraic surfaces
\cite{Wormleighton2022,ChaidezWormleighton2021}.

For closed contact three-manifolds, Hutchings also defined
elementary spectral invariants and used them to prove
quantitative closing lemmas for Reeb flows
\cite{Hutchings2024}. Their definition does not use the
$\mathrm{ECH}\cong\mathrm{HM}$ isomorphism. For arbitrary
closed contact three-manifolds, the proof of their Weyl law
uses a comparison with ECH and hence Seiberg--Witten theory
\cite[Section~6]{Hutchings2024}.

The Weyl laws for ECH spectral invariants and the volume
spectrum also have applications to density problems.
Irie used the Weyl law for ECH spectral invariants to prove
a $C^\infty$ closing lemma for three-dimensional Reeb flows
\cite[Lemma~3.1]{Irie2015}. In particular, for a
$C^\infty$-generic contact form on a closed contact
three-manifold, the union of periodic Reeb orbits is dense
\cite[Theorem~1.1]{Irie2015}.
Irie--Marques--Neves used the Weyl law for the volume spectrum
to prove a corresponding density result for minimal
hypersurfaces \cite{IMN2018}. They proved that, for a
$C^\infty$-generic metric on a closed manifold of dimension
three through seven, the union of closed smooth embedded
minimal hypersurfaces is dense. Marques--Neves--Song later
proved equidistribution of minimal hypersurfaces for generic
metrics in these dimensions \cite{MNS2019}.

Both density proofs use local perturbations and the Weyl law.
One changes the contact form or the metric in an open set to
change the volume. The Weyl law then forces some spectral
values to change \cite{Irie2015,IMN2018}.
These results naturally lead to the following question.

\begin{question}
The density of periodic Reeb orbits follows from the Weyl law
for ECH spectral invariants, whose proof uses Seiberg--Witten
theory. The density of minimal hypersurfaces follows from the
Weyl law for the volume spectrum defined by Almgren--Pitts
sweepouts.
Do these two results follow from a common underlying principle?
\end{question}

To study this question, we recall the Allen--Cahn construction
of minimal hypersurfaces. It gives a formulation of the volume
spectrum in terms of an energy functional. This construction
uses real-valued functions instead of cycles.
Let $(Y^n,g)$ be a closed connected Riemannian manifold with
$n\ge3$. The Allen--Cahn energy and the equation for its
critical points are given by
\[
 \begin{aligned}
 E^{\mathrm{AC}}_{\eps,g}(v)
 &=\int_Y\left(\frac\eps2|dv|_g^2+
                      \frac{W(v)}\eps\right)dV_g,\\
 -\eps^2\Delta_gv+W'(v)&=0,
 \qquad \Delta_g=\diver_g\nabla_g.
 \end{aligned}
\]
Here $W:\R\to[0,\infty)$ is a smooth even double-well
potential with nondegenerate minima at $\pm1$ and
$W(\pm1)=0$. A standard example is
\[
 W(v)=\frac14(1-v^2)^2.
\]
The energy of one transition is
\[
 \sigma_W=\int_{-1}^1\sqrt{2W(s)}\,ds.
\]
Modica--Mortola and Modica proved the relation between these
energies and perimeter by $\Gamma$-convergence
\cite{ModicaMortola1977,Modica1987}.
Hutchinson--Tonegawa studied the varifold limits of critical
points \cite{HutchinsonTonegawa2000}, and Guaraco used
min--max for this energy to obtain minimal hypersurfaces
\cite{Guaraco2018}. Gaspar--Guaraco used a cohomological index
for the symmetry $v\mapsto-v$ to define higher min--max values
$c_\eps(p)$ \cite{GasparGuaraco2018}. They then proved a
Weyl law for the phase-transition spectrum
\cite{GasparGuaraco2019}.

Dey proved that the Allen--Cahn spectrum agrees with the
volume spectrum after normalization
\cite[Theorems~1.1 and~1.2]{Dey2022}.
He proved the reverse of the comparison inequality obtained
in the earlier Allen--Cahn papers. The two inequalities give,
for each fixed $p$,
\[
 \lim_{\eps\downarrow0}\frac{c_\eps(p)}{\sigma_W}
 =\omega_p(Y,g).
\]
Our $\sigma_W$ is $2\sigma$ in the notation of
\cite{Dey2022}. Thus min--max for the Allen--Cahn energy gives
the same widths as min--max for cycles. These are the widths
used in the density theorem of minimal hypersurfaces.

Pigati--Stern studied the self-dual Abelian--Higgs model as
a codimension-two version of the Allen--Cahn theory
\cite{PigatiStern2021}. The variables are a section of a
Hermitian line bundle and a unitary connection.
They proved a convergence result on a fixed closed oriented
Riemannian manifold. For a sequence of critical points with
uniformly bounded energy, a subsequence of the normalized
energy measures converges as $\eps\to0$ to the weight measure
of a stationary integral codimension-two varifold
\cite[Theorem~1.1]{PigatiStern2021}.
They also constructed nontrivial critical points on the
trivial bundle by min--max and proposed higher min--max
constructions using the topology of the gauge quotient
\cite[Section~7]{PigatiStern2021}.

Parise--Pigati--Stern proved $\Gamma$-convergence of the
normalized AH energies to the mass of integral cycles in
codimension two \cite[Theorem~1.2]{ParisePigatiStern2024}.
They also compared AH min--max values with Almgren--Pitts
widths defined using cycles with integer coefficients.
Their comparison concerns homotopy classes represented by
families with one or two parameters. For these classes,
the Almgren--Pitts widths are bounded above by the lower
limits of the AH min--max values as $\eps\downarrow0$,
with the normalization in \eqref{eq:unitary-AH-energy}
\cite[Theorem~5.5]{ParisePigatiStern2024}.
This comparison does not identify the higher AH spectrum
with the spectrum of integral cycles.

For a fixed metric, the AH energy is therefore related to
minimal submanifolds of codimension two. In dimension four,
the limiting varifolds are two-dimensional.
Taubes used planar vortices to relate Seiberg--Witten
solutions to pseudoholomorphic curves
\cite{Taubes1980,Taubes1996}.
These facts suggest that the AH energy can be used to study
the relation between minimal surface theory and
Seiberg--Witten theory.

In this paper, we use powers of an evaluation class of
degree two with integer coefficients to define higher AH
min--max values. We then take suprema over domains and
compatible almost complex structures to obtain a sequence
of symplectic capacities.
To estimate these values, we construct AH sweepouts from
Seiberg--Witten moduli spaces. The evaluation class on these
moduli spaces gives the topological condition for the
sweepouts, and symplectic areas control their AH energies.
These families give upper bounds for the capacities of
balls and cylinders. Localization and the normalization
of the first capacity then give the exact formula for
the cylinder. We combine these estimates with localization
to prove the Weyl laws for domains in $\R^4$.

These results give a concrete connection between the
variational theory of minimal surfaces and the
Seiberg--Witten theory used in Reeb dynamics.
To study this connection further, we ask whether the AH
capacities agree with Hutchings's alternative capacities
and ECH capacities.
Write $c_k^{\mathrm{Alt}}$ for Hutchings's alternative
capacities. These alternative capacities need not agree
with ECH capacities on all Liouville domains
\cite[Remark~14]{Hutchings2022}.
We state the following conjecture for star-shaped domains
in $\R^4$.

\Needspace{9\baselineskip}
\begin{conjecture}\label{conj:capacity-comparison}
Let $X\subset\R^4$ be a compact star-shaped domain with smooth
boundary. Then, for every $k\ge1$,
\[
 \cAH_k(X,\omega_0)
 =c_k^{\mathrm{Alt}}(X,\omega_0)
 =c_k^{\mathrm{ECH}}(X,\omega_0).
\]
\end{conjecture}

Hutchings proved the equality of the last two capacities for
convex and concave toric domains \cite{Hutchings2022}.

\subsection{Outline of the proofs}\label{sec:intro-proofs}
In Section~\ref{c:sec:topology}, we prepare the properties of the
evaluation class and construct projective families. In
Section~\ref{c:sec:cylinder}, we use these families to prove the
upper bound on the cylinder. We combine Gromov's family of
$J$-holomorphic spheres with Taubes's planar vortex. The form of
Gromov's result \cite{Gromov1985} which we use is stated in
Theorem~\ref{c:thm:gromov-ruling}. We place a vortex in the normal
directions to each sphere. We then apply the symmetric product
construction of Section~\ref{c:sec:products} to obtain the
families needed for the upper bound.

For the lower bound on the ball, we first approximate the metric
on the doubled ball by smooth metrics. Allard compactness and
Brendle's vector field give a sharp limiting bound for the mass
\cite{Allard1972,Brendle2012}. By using the closed AH theory, we
obtain a gap in the positive critical values. We then use an
equivariant deformation to trivialize the evaluation line on a low
sublevel. Finally, we transfer this trivialization to the ball by
reflection and a gauge correction
(Sections~\ref{c:sec:double} and~\ref{c:sec:closed-AH}). These
estimates give the normalization of the first capacity. We use it to prove the
capacity axioms in Section~\ref{c:sec:capacity}.

The proof through Section~\ref{c:sec:capacity} does not use
Seiberg--Witten theory. The cylinder families give the bound
$\cAH_k(Z^4(a))\le k^2a$, which is enough for the capacity axioms.

For the higher estimates and the Weyl laws, we first prove
localization in Section~\ref{c:sec:localization}. We use relative
cup products on finite parameter complexes for a fixed metric.
We then obtain the localization inequality for the capacities.
In Section~\ref{m:sec-comparison}, Taubes's estimates give
\[
 \En^X_{\eps_r,g}(\alpha,a)=c_1(E)\cdot[\omega]+o(1)
\]
for the scalar component of a Seiberg--Witten solution.
In Section~\ref{m:sec-upper}, we use regular moduli spaces on
$\CP^2$ and $S^2\times S^2$. The top-degree powers of their evaluation classes have
nonzero pairings with their integral fundamental classes. We restrict the fields to domains on which the line
bundle is trivial. This gives AH sweepouts and proves the ball
and cylinder upper bounds. Localization and the normalization of the first capacity give the opposite bound on the cylinder. Together,
these prove Theorem~\ref{m:thm:models}.

In Section~\ref{m:sec-weyl}, we prove the two Weyl laws together.
Both sequences satisfy the scaling and localization properties
needed for the subdivision argument. We first prove the limit on
the unit cube. We then prove the limit on smooth domains. For the
upper bound, we use localization in an outer cube and subtract the
contributions of cubes in the complement. Thus we do not need to
glue arbitrary AH families. This argument gives the domain limit
for the standard metric as well as for the capacities. For the
capacities, we also use inner and outer approximation to obtain
the bounds on general connected open sets.

Finally, comparison on inner cubes gives
$b_{\AH}^{\mathrm{flat}}\le b_{\AH}$. The upper bound on the ball
gives $b_{\AH}\le2$, and symplectic packings by equal balls give
$b_{\AH}\ge\sqrt2$ \cite{McDuff2009,McDuffPolterovich1994}.
These estimates complete the proofs of
Theorems~\ref{thm:intro-weyl} and~\ref{thm:intro-fixed-weyl}.

\paragraph{Acknowledgements.}
The author is particularly grateful to Alessandro Pigati for pointing out that the author's original argument for the ball upper bound also applies to the cylinder. This observation led to the sharp cylinder estimate in Theorem~\ref{m:thm:models}. Thanks are also due to Daniel Stern for his comments and suggestions, Jin Miyazawa for his interest in this work, and Michael Hutchings for his interest and comments.

\paragraph{Use of AI.}
The author used ChatGPT6 Astra as an assistive tool in preparing this paper. The author used it to discuss and improve the clarity and English presentation of his original ideas and proofs, and to check for errors in spelling and notation. Regarding the mathematical content, the central idea of taking the double of the ball and applying Brendle's vector field to prove the lower bound for the symplectic capacity of the ball in Sections 4 and 5 was suggested by it.

\section{The evaluation class and projective families}\label{c:sec:topology}

\subsection{Energy bounds and scaling}
\begin{lemma}\label{c:lem:elementary}
For every $\eps>0$ and $k\ge1$, we have
\begin{equation}\label{c:eq:raw-upper-barrier}
 0\le w_{\eps,k}(D,g)\le\frac{\Vol_g(D)}{8\pi\eps^2}.
\end{equation}
Both $w_{\eps,k}(D,g)$ and $W_k(D,g)$ are nondecreasing in $k$.
\end{lemma}
\begin{proof}[\bf Proof of Lemma \ref{c:lem:elementary}]
We choose linearly independent smooth complex functions $s_0,\ldots,s_k$ on
$\overline D$. For $\delta>0$, we consider
\[
 S^{2k+1}\longrightarrow\Conf^*(D),\qquad
 z\longmapsto\left(\delta\sum_{j=0}^kz_js_j,0\right).
\]
This map is equivariant under constant gauge transformations and induces
$F_\delta:\CP^k\to\B^*(D)$. The weight-one convention in
\eqref{c:eq:eval-line} gives $F_\delta^*\mu_D=h=c_1(\mathcal O(1))$.
Thus, after choosing a finite triangulation of $\CP^k$, we obtain an
AH $k$-sweepout.

For fixed $\eps,k$, the energy converges uniformly on the family to
$\Vol_g(D)/(8\pi\eps^2)$ as $\delta\to0$. By taking the infimum, we obtain \eqref{c:eq:raw-upper-barrier}.
If $(F^*\mu_D)^{k+1}\ne0$, then $(F^*\mu_D)^k\ne0$. Hence every $(k+1)$-sweepout is also a
$k$-sweepout. This proves monotonicity at each fixed scale.
Taking the lower limit gives the other monotonicity statement.
\end{proof}

We will also use
\begin{equation}\label{c:eq:zero-barrier}
 \En^D_{\eps,g}(0,A)\ge\frac{\Vol_g(D)}{8\pi\eps^2}.
\end{equation}
We note that this estimate concerns a section which is identically zero. We still exclude the zero section
from $\B^*(D)$. In \eqref{c:eq:raw-upper-barrier}, it only appears as the limit of a family of
nonzero sections. The finite scale upper bound tends to infinity as
$\eps\to0$.

If $\mu_D^k$ vanishes on $\{\En^D_{\eps,g}<L\}$, then $w_{\eps,k}(D,g)\ge L$. In particular, if the
evaluation line is trivial on this sublevel, the latter bound holds for
every $k\ge1$.

Let $\varphi:D\to D'$ extend to a smooth diffeomorphism of the closures. By
pullback, the configuration quotients, the energies for the
corresponding metrics, and the evaluation classes are identified. The
condition $F^*(\mu^k)\ne0$ is preserved, and we have
\begin{equation}\label{c:eq:naturality}
 w_{\eps,k}(D,\varphi^*g')=w_{\eps,k}(D',g'),\qquad
 W_k(D,\varphi^*g')=W_k(D',g').
\end{equation}

For the energy on a Hermitian line bundle over an $n$-dimensional
Riemannian manifold, a constant change of the metric gives
\begin{equation}\label{eq:general-metric-scaling}
 \En_{\eps,\lambda^2g}(u,a)
 =\lambda^{n-2}\En_{\eps/\lambda,g}(u,a),\qquad \lambda>0.
\end{equation}
In particular, the energy has the scale of area in dimension four.
For $t>0$, we set $n=4$ and $\lambda=\sqrt t$ in this identity and obtain
\begin{equation}\label{c:eq:scale-metric}
 \En_{\eps,tg}=t\En_{\eps/\sqrt t,g}.
\end{equation}
Neither the configuration quotient nor the sweepout condition depends
on the metric. Therefore
\begin{equation}\label{eq:scale-family}
 w_{\eps,k}(D,tg)=t\,w_{\eps/\sqrt t,k}(D,g),\qquad
 W_k(D,tg)=tW_k(D,g).
\end{equation}
For $D \subset \R^{4}$, we have
\begin{equation}\label{c:eq:scale-domain}
 w_{\eps,k}(rD,g_0)=r^2w_{\eps/r,k}(D,g_0),\qquad
 W_k(rD,g_0)=r^2W_k(D,g_0).
\end{equation}

\subsection{Projective families}\label{c:sec:products}
For the projective families below, we use $*=[(1,0)]$ and $h=c_1(\mathcal O(1))$.

\begin{lemma}\label{c:lem:disk}
Fix $x_0\in D$.
Let $\widetilde F:\overline{\mathbb D}\to\Conf^*(D)$ be a continuous family of smooth configurations. Suppose
that
\[
 \widetilde F(e^{i\theta})=\gamma_\theta\cdot(1,0),\qquad
 \gamma_\theta\in\G(D),
\]
and that $\theta\mapsto\gamma_\theta(x_0)$ has degree one. Then the quotient family induces a map $F:\CP^1\to\B^*(D)$ such that $F^*\mu_D=h$.
\end{lemma}
\begin{proof}[\bf Proof of Lemma \ref{c:lem:disk}]
Every boundary configuration represents $*=[(1,0)]$. Therefore the
quotient family descends to
$\overline{\mathbb D}/\partial\mathbb D\cong S^2$.
By local triviality of the gauge projection, we can choose a local
representative near $*=[(1,0)]$ whose value at $*$ is $(1,0)$. Its pullback gives a
trivialization near the collapsed boundary point. The original lift
gives a trivialization on the parameter disk. The transition function
of the evaluation line is $\gamma_\theta(x_0)$, and its winding number is one.
With the above sign convention, we obtain $F^*\mu_D=h$.
\end{proof}

Although the quotient family is constant on the parameter boundary, the
representatives may have nonzero winding number. The vortex
family constructed in Section~\ref{c:sec:vortex} satisfies the assumptions of
Lemma~\ref{c:lem:disk}.

We now construct a $\CP^k$-family from a $\CP^1$-family by taking
products of configurations. The parameter space is $\Sym^k(\CP^1)\cong\CP^k$.

\begin{lemma}\label{c:lem:product}
Fix $\eps>0$ and $k\ge1$. Suppose that $F:\CP^1\to\B^*(D)$ is continuous,
$F(z_0)=*$, and $F^*\mu_D=h$. Suppose also that the sections of this family
satisfy $|u|\le b$, where $b\ge1$, and that the maximum energy is at most
$\Lambda$. If every product $\prod_{i=1}^k u_i$ of $k$ sections $u_1,u_2,\ldots,u_k$ in the family is not
identically zero, then there exists a continuous map $F^{(k)}:\CP^k\to\B^*(D)$ satisfying
\begin{equation}\label{eq:projective-trial-bound}
 (F^{(k)})^*\mu_D=h,\qquad
 \max_{\CP^k}\En_{\eps,g}^D\circ F^{(k)}
 \le k^2b^{4(k-1)}\Lambda.
\end{equation}
Consequently, $w_{\eps,k}(D,g)\le k^2b^{4(k-1)}\Lambda$.
\end{lemma}
\begin{proof}[\bf Proof of Lemma \ref{c:lem:product}]
Choose representatives $(u_1,A_1),(u_2,A_2),\ldots,(u_k,A_k)$ of $F(z_i)$ and define
\[
 u=\prod_{i=1}^k u_i,\qquad A=\sum_{i=1}^k A_i.
\]

Replacing each $(u_i,A_i)$ by $\gamma_i\cdot(u_i,A_i)$
changes $(u,A)$ by the gauge transformation
$\prod_{i=1}^k\gamma_i$.
Thus the product is well defined on the gauge quotient.

We have
\[
 d_Au=\sum_i\left(\prod_{j\ne i}u_j\right)d_{A_i}u_i,
 \qquad dA=\sum_i dA_i,
\]
By the Cauchy--Schwarz inequality, we have
\begin{align}
 |d_Au|^2&\le kb^{2(k-1)}\sum_i|d_{A_i}u_i|^2,\notag\\
 |dA|^2&\le k\sum_i|dA_i|^2,\notag\\
 \left(1-\prod_i|u_i|^2\right)^2
 &\le kb^{4(k-1)}\sum_i(1-|u_i|^2)^2.
 \label{c:eq:product-components}
\end{align}
Note that the third inequality follows from $
 1-\prod_i t_i=\sum_i\left(\prod_{j<i}t_j\right)(1-t_i).
$

After integration, the energy is at most $kb^{4(k-1)}\sum_i\En_{\eps,g}^D(u_i,A_i)$. This gives 
the bound in \eqref{eq:projective-trial-bound}. 
The product family is continuous and symmetric in the factors.
Therefore it descends to $\Sym^k(\CP^1)\cong\CP^k$. The map $z\mapsto z+(k-1)z_0$ parametrizes a
projective line. Therefore, restriction to a
projective line identifies the positive generators of second
cohomology. Hence $(F^{(k)})^*\mu_D=h$. The $k$th power is $h^k\ne0$. After
choosing a finite triangulation of $\CP^k$, we can regard $F^{(k)}$ as
a $k$-sweepout. This proves the min--max bound.
\end{proof}

In the foliation construction below, the sections have modulus at most
one, and their zero sets are nowhere dense. Thus we can apply the lemma
with $b=1$ and obtain $k^2\Lambda$.

\section{An upper bound for the cylinder}\label{c:sec:cylinder}
In this section, we prove an upper bound for the capacities
of the cylinder using Gromov's family of $J$-holomorphic
spheres and Taubes's degree-one planar vortex. Gromov's theorem gives a smooth family of spheres whose
areas, for the compatible metric, are determined by their homology
class. We place a vortex in the two normal directions to each sphere.
Since the normalized energy of the planar vortex is one, the
AH energy approaches the area of the sphere. We also check that the
family has the required nonzero evaluation Chern class. The
symmetric-product construction of Section~\ref{c:sec:products} then gives
the higher sweepouts.

More precisely, for every connected smooth domain $D\Subset Z(a)$,
every  $\omega_0$-compatible structure $J$ near $\overline D$,
and every fixed $k\ge1$, we will prove
\[
 \limsup_{\eps\downarrow0}w_{\eps,k}(D,g_J)\le k^2a.
\]
We first state the form of Gromov's  result that we use in
this argument, and explain its relation to the original statement. We then
recall the planar vortex and construct the AH family.

\subsection{Gromov's family of \texorpdfstring{$J$}{J}-holomorphic spheres}\label{sec:gromov-ruling}
We use the following form of Gromov's  result. We state it
for a product of two spheres and for a compatible almost complex
structure, since this is the setting needed for the cylinder bound.

\begin{theorem}[{cf.\ \cite[Theorem~0.2.A and Section~2.4.B$_2$]{Gromov1985}}]\label{c:thm:gromov-ruling}
Let $S_1$ and $S_2$ be oriented surfaces diffeomorphic to $S^2$, with
area forms $\omega_1$ and $\omega_2$. Put
\[
 A_i=\int_{S_i}\omega_i>0,\qquad
 M_0=S_1\times S_2,\qquad
 \Omega=\operatorname{pr}_1^*\omega_1+
        \operatorname{pr}_2^*\omega_2,
\]
and suppose that $A_2=NA_1$ for an integer $N\ge1$. For every smooth
$\Omega$-compatible almost complex structure $J$, there is a smooth
fibration
\[
 \pi_J:M_0\longrightarrow S^2
\]
whose fibers are embedded $J$-holomorphic spheres in the class
\[
 F=[S_1\times\{\mathrm{pt}\}]\in H_2(M_0;\Z).
\]
Through each point of $M_0$, there is exactly one such sphere up to  reparametrization.
For the compatible metric $g_J=\Omega(\cdot,J\cdot)$, every fiber $\Sigma$
satisfies
\[
 \Area_{g_J}(\Sigma)=\int_\Sigma\Omega=A_1.
\]
\end{theorem}

We explain the relation with the original statement. In the first
assertion of \cite[Theorem~0.2.A, p.~308]{Gromov1985}, Gromov allows
$S_2$ to be any closed connected surface and assumes that $J$ is tamed
by $\Omega$. He proves the existence and uniqueness of an embedded
rational $J$-curve through a  point in the fiber class. Our statement restricts both factors to spheres and
assumes compatibility. It also includes the fibration conclusion from
\cite[Section~2.4.B$_2$, p.~338]{Gromov1985}.

For the fibration conclusion, put $G=[\{\mathrm{pt}\}\times S_2]$.
Then $\Omega(mF+nG)=(m+Nn)A_1$, so every nonconstant $J$-holomorphic
sphere has area at least $A_1$. Since $\Omega(F)=A_1$, the class $F$
cannot split into two nonconstant components. It is primitive and
satisfies $F\cdot F=0$. Thus its embedded representatives have normal
Chern number zero, and the fibration argument in Section~2.4.B$_2$
applies. The base is a closed simply connected surface, because both
the total space and the fibers are simply connected.  

No genericity assumption on $J$ is needed for this statement. 

\Needspace{10\baselineskip}
\subsection{Vortex families along a smooth foliation}\label{sec:planar-input}\label{c:sec:vortex}
We use Taubes's positive degree-one vortex
\cite{Taubes1980,TaubesEquivalence1980}, with the decay statement in
\cite[Theorem~4.5]{ParisePigatiStern2024}. On the Euclidean plane, write
$\nabla_a=d-iA$ and $D_j=\partial_{x_j}-iA_j$. With the positive sign
convention, the vortex equations at scale $\eps>0$ are
\begin{equation}\label{eq:planar-vortex}
 D_1u+iD_2u=0,\qquad i*F_a=*dA=\frac{1-|u|^2}{2\eps^2}.
\end{equation}
At scale one, write the pair as $(v,\alpha)$, where $\alpha$ is a real
one-form and the unitary connection satisfies $\nabla_a=d-i\alpha$.
We use the following properties of this solution:
\begin{enumerate}[label=\textup{(\roman*)}]
\item The gauge class is uniquely determined by its single simple zero.
We place this zero at $0\in\C$.
\item The section satisfies $|v|\le1$ and $v^{-1}(0)=\{0\}$. Its
unnormalized and normalized energies are
\begin{equation}\label{c:eq:planar-input}
 E_1^{\R^2}(v,\alpha)=2\pi,\qquad
 \En_1^{\R^2}(v,\alpha)=1.
\end{equation}
\item The total curvature is
\begin{equation}\label{eq:planar-flux-input}
 \int_{\C}d\alpha=\int_{\C}iF_a=2\pi.
\end{equation}
\item There are $C,c>0$ such that
\begin{equation}\label{c:eq:planar-decay}
 |d_\alpha v|+(1-|v|^2)+|d\alpha|\le Ce^{-c|y|}.
\end{equation}
We can choose the gauge so that $v=\rho e^{i\theta}$ outside the unit
disk. In this gauge, $1-\rho$, $d\rho$, $\alpha-d\theta$, and
$d\alpha$ decay exponentially.
\end{enumerate}
The solution at scale $\eps$ is obtained by pullback under
$s_\eps(y)=y/\eps$. The flux and energy in
\eqref{c:eq:planar-input}--\eqref{eq:planar-flux-input} are unchanged.

We also use a cutoff of this solution. The cutoff calculation is the
one used in the proof of \cite[Proposition 4.4]{ParisePigatiStern2024}.
We record only its conclusion. For every sufficiently large $L$, we
choose a smooth pair $(v_L,\alpha_L)$ which agrees with $(v,\alpha)$
on $|y|\le L$ and satisfies
\begin{equation}\label{c:eq:planar-cutoff}
 |v_L|\le1,\qquad v_L^{-1}(0)=\{0\},\qquad
 (v_L,\alpha_L)=(e^{i\theta},d\theta)\quad (|y|\ge2L).
\end{equation}
Moreover,
\begin{equation}\label{c:eq:tau}
 \tau_L:=E_1^{\R^2}(v_L,\alpha_L)\longrightarrow2\pi.
\end{equation}
These cutoff pairs are configurations, not exact vortex solutions.
Put $v_{\eps,L}=v_L\circ s_\eps$ and
$\alpha_{\eps,L}=s_\eps^*\alpha_L$. Their unnormalized energy densities
satisfy
\begin{equation}\label{c:eq:planar-moments}
 \begin{split}
 e_{\eps,L}(y)&=\eps^{-2}e_{1,L}(y/\eps),\qquad
 \supp e_{\eps,L}\subset\{|y|\le2L\eps\},\\
 \int_{\R^2}e_{\eps,L}\,dy&=\tau_L,\qquad
 \int_{\R^2}|y|e_{\eps,L}(y)\,dy\le2L\tau_L\eps.
 \end{split}
\end{equation}

We now use them to construct a family along the fibers
of a smooth submersion.

\begin{lemma}\label{c:lem:foliation}
Let $p:V\to\C$ be a smooth proper submersion from a Riemannian
four-manifold $(V,g)$. Its fibers $\Sigma_z=p^{-1}(z)$ are closed surfaces. Let
$D\Subset V$ be connected and smooth. Suppose that the following bound
holds on a neighborhood of a closed parameter disk containing
$p(\overline D)$:
\[
 \Area_g(\Sigma_z)\le A_0.
\]
Then there exist based continuous maps $F_\eps:\CP^1\to\B^*(D)$ satisfying
\begin{equation}\label{c:eq:foliation-result}
 F_\eps^*\mu_D=h,\qquad
 \limsup_{\eps\downarrow0}\max_{\CP^1}\En^D_{\eps,g}\circ F_\eps\le A_0.
\end{equation}
Their sections have modulus at most one. Each zero set is the
restriction of one fiber or the empty set. For each fixed $k$,
the symmetric products give based maps from $\CP^k$ whose pullback
class is $h$ and whose maximum energy is at most $k^2A_0+o(1)$.
In particular,
\begin{equation}\label{c:eq:foliation-higher}
 \limsup_{\eps\downarrow0}w_{\eps,k}(D,g)\le k^2A_0.
\end{equation}
\end{lemma}
\begin{proof}[\bf Proof of Lemma \ref{c:lem:foliation}]
We choose $R_0>\sup_{\overline D}|p|$ and $\rho_0>0$ such that the fiber-area bound holds
for $|z|\le R_0+\rho_0$. Let
\[
K=p^{-1}(\{\zeta\in\C:|\zeta|\le R_0+\rho_0\}).
\]
We identify each tangent space of $\C$ with $\R^2$ naturally,
and write $\langle\cdot,\cdot\rangle_0$ for its Euclidean inner product.
For $x\in V$, the differential and its pointwise adjoint are real-linear maps
\[
 dp_x:T_xV\longrightarrow\R^2,\qquad
 dp_x^*:\R^2\longrightarrow T_xV.
\]
The adjoint is defined by
\[
 g_x(v,dp_x^*\eta)=\langle dp_xv,\eta\rangle_0
 \qquad(v\in T_xV,\ \eta\in\R^2).
\]
We now put
\begin{equation}\label{c:eq:transverse-normalization}
 G(x)=dp_x\circ dp_x^*,\qquad B(x)=G(x)^{-1/2},\qquad
 w_z(x)=B(x)(p(x)-z),\quad |z|\le R_0.
\end{equation}
Thus $G(x)\in\operatorname{End}_{\R}(\R^2)$ is a real $2\times2$
symmetric positive-definite matrix. Indeed, $dp_x$ is surjective, and
\[
 \langle G(x)\eta,\eta\rangle_0=|dp_x^*\eta|_g^2>0
 \qquad(\eta\ne0).
\]
If $p=p_1+ip_2$, then its matrix entries are
$G_{ij}(x)=g_x(\nabla_g p_i,\nabla_g p_j)$, for $i,j\in\{1,2\}$.
The inverse square root $B(x)\in\operatorname{End}_{\R}(\R^2)$ is also
symmetric positive definite and depends smoothly on $x$. 

For each fixed $z$, we use the same convention for $w_z:V\to\R^2$:
\[
 (dw_z)_x:T_xV\longrightarrow\R^2,\qquad
 (dw_z)_x^*:\R^2\longrightarrow T_xV,
\]
where
\[
 g_x(v,(dw_z)_x^*\eta)=\langle(dw_z)_xv,\eta\rangle_0
 \qquad(v\in T_xV,\ \eta\in\R^2).
\]
 We write $dw_z^*$ for the family
of maps $(dw_z)_x^*$.

On $\Sigma_z$, the term containing the derivative of $B$ vanishes.
Since $B(x)^*=B(x)$, we obtain, for $x\in\Sigma_z$,
\begin{equation}\label{c:eq:coisometry}
 (dw_z)_x=B(x)dp_x,\qquad
 (dw_z)_x(dw_z)_x^*=B(x)G(x)B(x)=\Id_{\R^2}.
\end{equation}
Here $\ker(dw_z)_x=T_x\Sigma_z$. Hence the restriction of $(dw_z)_x$
to the normal plane $(T_x\Sigma_z)^{\perp_g}$ is an isometry onto
$\R^2$. 

At an arbitrary point, differentiation gives
\[
 (dw_z)_x(v)=B(x)dp_x(v)+(dB)_x(v)(p(x)-z)
 \qquad(v\in T_xV).
\]
Define the real-linear endomorphism
$H_z(x)=(dw_z)_x\circ(dw_z)_x^*\in\operatorname{End}_{\R}(\R^2)$.
Since $p-z=B^{-1}w_z$ and the fixed coefficients are bounded on $K$,
\begin{equation}\label{c:eq:H-near}
 H_z(x)=\Id_{\R^2}+O(|w_z(x)|)
\end{equation}
uniformly for $|z|\le R_0$ near the zero fibers. The error is measured
in the Euclidean operator norm.

The set $\{(x,z):p(x)=z,\ |z|\le R_0\}$ is compact. We apply the implicit function theorem in
the normal directions, with parameter $z$. For sufficiently small $y$, the set $K\cap w_z^{-1}(y)$ is a normal graph over $\Sigma_z$. The graph and its first derivative are $O(|y|)$, uniformly for $|z|\le R_0$. Thus
\begin{equation}\label{c:eq:level-area}
 \Area_g(K \cap w_z^{-1}(y))=\Area_g(\Sigma_z)+O(|y|).
\end{equation}
For small $y$, these graphs contain the whole level sets in
$K$. 

Fix $L$. From this point on, $w_z^*$ denotes pullback of differential
forms by the map $w_z$. We pull back the planar pair and define
\begin{equation}\label{c:eq:foliation-pair}
 u_{\eps,z}=v_{\eps,L}\circ w_z,\qquad
 A_{\eps,z}=w_z^*\alpha_{\eps,L}.
\end{equation}
Covariant differentiation and exterior differentiation commute with
this pullback. Therefore
\[
 d_{A_{\eps,z}}u_{\eps,z}=w_z^*(d_{\alpha_{\eps,L}}v_{\eps,L}),
 \qquad dA_{\eps,z}=w_z^*(d\alpha_{\eps,L}).
\]
On the set where the energy density is nonzero, we have $|w_z|\le2L\eps$. This
set lies in the common tube for sufficiently small $\eps$.

Let $\xi$ be a real one-form on $\R^2$, and let $\xi^\sharp$ be its
Euclidean metric dual. For $y=w_z(x)$, we have
\[
 |(w_z^*\xi)_x|_g^2
 =\langle H_z(x)\xi_y^\sharp,\xi_y^\sharp\rangle_0.
\]
For a complex-valued one-form, the corresponding identity follows by
applying this formula to its real and imaginary parts. For a planar
two-form $\zeta$, we have
\[
 |(w_z^*\zeta)_x|_g^2=(\det H_z(x))|\zeta_y|_0^2.
\]
The coarea Jacobian at $x$ is $\sqrt{\det H_z(x)}$. By \eqref{c:eq:H-near},
the metric corrections to these norm and Jacobian factors are
$1+O(|y|)$. Combining this with \eqref{c:eq:planar-moments} and
\eqref{c:eq:level-area}, we obtain
\begin{align}
 E^K_{\eps,g}(u_{\eps,z},A_{\eps,z})
 &\le\int_{\R^2}e_{\eps,L}(y)
       \bigl(\Area_g(\Sigma_z)+C|y|\bigr)\,dy\notag\\
 &\le \tau_L\Area_g(\Sigma_z)+C_L\eps.\label{c:eq:coarea-bound}
\end{align}
Here we use the coarea formula for a smooth submersion; see \cite[Chapter~2, Section~6]{Simon}.
Restriction from $K$ to $D$ does not increase the energy.
We first proved the estimate for the closed fibers.

Next, we calculate the pullback of $\mu_D$. If $z=R_0e^{i\theta}$, then
$|p(x)-z|$ has a positive lower bound for $x\in\overline D$. Since $B$ is
uniformly invertible, $|w_z|$ also has a positive lower bound there.
For sufficiently small $\eps$, the restricted pair is exactly a pure
gauge with phase
\begin{equation}\label{c:eq:foliation-phase}
 q_\theta(x)=\frac{B(x)(p(x)-R_0e^{i\theta})}
 {|B(x)(p(x)-R_0e^{i\theta})|}.
\end{equation}
Its gauge class is the vacuum. Hence we can collapse the parameter
boundary and obtain a based sphere family.

For $0\le s\le1$, let $B_s=(1-s)B+s\Id$. We consider the homotopy
\[
 \frac{B_s(x)((1-s)p(x)-R_0e^{i\theta})}
 {|B_s(x)((1-s)p(x)-R_0e^{i\theta})|}
\]
It is defined everywhere, since $B_s$ is positive definite and
$|(1-s)p(x)|<R_0$. At the endpoint, it is $-e^{i\theta}$. Therefore the evaluation
loop at every $x_0\in D$ has degree one. By the clutching calculation in
Lemma~\ref{c:lem:disk}, we have $F_{\eps,L}^*\mu_D=h$.

The zero set of \eqref{c:eq:foliation-pair} is $D\cap\Sigma_z$. Thus the section is not
identically zero. All operations are continuous in $z$ in the
smooth topology. For fixed $L$, we have
\[
 \max\En^D_{\eps,g}\circ F_{\eps,L}
 \le\frac{\tau_L}{2\pi}A_0+C_L\eps.
\]
We first let $\eps\downarrow0$ with $L$ fixed, and then let $L\to\infty$.
We can choose $L=L(\eps)\to\infty$ sufficiently slowly and write $F_\eps$ for the
 family. This proves \eqref{c:eq:foliation-result}. We do not need the estimates
to be uniform in $L$.

For \eqref{c:eq:foliation-higher}, we apply the symmetric-product construction in
Lemma~\ref{c:lem:product}. Since $|u_i|\le1$, the product energy is at most
$k\sum_i\En_{\eps,g}(u_i,A_i)$. Each zero set is closed and has empty interior. The finite
union of these zero sets cannot contain $D$. Hence the product
section is not identically zero. By restricting to the projective line
$z+(k-1)z_0$, we see that the pullback class is $h$. Therefore the
upper limit of the maximum energy is at most $k^2A_0$.
\end{proof}

\subsection{Proof of the upper bound}\label{sec:cylinder-proof}
\begin{lemma}\label{c:prop:cylinder}
For every connected smooth domain $D\Subset Z(a)$,
every compatible $J$ near $\overline D$,
and every fixed $k\ge1$, we have
\begin{equation}\label{c:eq:cylinder-bound}
 \limsup_{\eps\downarrow0}w_{\eps,k}(D,g_J)\le k^2a.
\end{equation}
Consequently, $\cAH_k(Z(a))\le k^2a$.
\end{lemma}
\begin{proof}[\bf Proof of Lemma \ref{c:prop:cylinder}]
Fix $D,J,k$, and choose $a_0>a$. We take area forms $\omega_1$ and
$\omega_2$ on $S^2$ with total areas $a_0$ and $Na_0$, respectively,
where the integer $N\ge2$ is chosen sufficiently large. A product of
symplectic disk charts in
\[
 M_0=S^2(a_0)\times S^2(Na_0),\qquad
 \Omega=\operatorname{pr}_1^*\omega_1+
        \operatorname{pr}_2^*\omega_2
\]
then contains a neighborhood of $\overline D$. This is possible because
$D\Subset Z(a)$, so the projection of $\overline D$ to the second
coordinate plane is bounded. Here $S^2(b)$ denotes a sphere with an
area form of total area $b$. We identify $D$ and $J$ with their images
in this symplectic chart.

The sphere $\Sigma_\infty=S^2(a_0)\times\{\infty\}$ does not meet the
chosen chart.  We extend an $\Omega$-compatible structure
$\widetilde J$ on $M_0$ which agrees with the prescribed $J$ near
$\overline D$ and with the product structure near $\Sigma_\infty$.
These two compact subsets are disjoint. Its compatible metric
$g_{\widetilde J}=\Omega(\cdot,\widetilde J\cdot)$ pulls back to the
original $g_J$ near $\overline D$.

The hypotheses of Theorem~\ref{c:thm:gromov-ruling} are now satisfied. By the theorem, $M_0$ has a smooth fibration by
embedded $\widetilde J$-holomorphic spheres in the class
$F=[S^2\times\{\mathrm{pt}\}]$. Since $\Sigma_\infty$ is already such
a sphere, uniqueness implies that it is one of the fibers.

We remove this fiber and identify the remaining base with $\C$. Thus
we obtain a smooth proper submersion
\[
 p:M_0\setminus\Sigma_\infty\longrightarrow\C.
\]
The inverse image of a
compact subset of the remaining base is compact and disjoint from
$\Sigma_\infty$. Each fiber is $\widetilde J$-holomorphic and satisfies
\begin{equation}\label{c:eq:calibrated-fiber}
 \Area_{g_{\widetilde J}}(\Sigma_z)=\int_{\Sigma_z}\Omega=a_0.
\end{equation}

We apply Lemma~\ref{c:lem:foliation} with $A_0=a_0$. The vortex
construction gives projective families on $D$ with the required
pullback class and with
$\limsup_{\eps\downarrow0}w_{\eps,k}(D,g_J)\le k^2a_0$.
We let $a_0\downarrow a$ while keeping $D,J,k$ fixed. This proves
\eqref{c:eq:cylinder-bound}. Since the right side is independent of
$D,J$, we can then take their suprema. 
\end{proof}

\section{A mass estimate on the doubled ball}\label{c:sec:double}
We prove a lower bound for the mass of nonzero stationary integral
two-varifolds on smooth approximations to the double of a Euclidean
ball. The estimate is stated in Lemma~\ref{c:prop:double-gap}.
In Section~\ref{c:sec:closed-AH}, we use it to exclude positive AH
critical values below a fixed level on a smooth double and to prove
the ball lower bound.

\subsection{Integral varifolds and compactness}\label{x:gmt-interface}
The limiting varifold in the theorem of Pigati--Stern has dimension
$n-2$; see Section~\ref{sec:intro-background}. We use $n=4$ and
recall the compactness and density results for integral two-varifolds
needed below. These follow from Allard's compactness theorem and the
monotonicity formula \cite{Allard1972}. We use the statements in
\cite[Chapters 4 and 8]{Simon} and \cite[Theorem 3.2]{White}.

Let $(M,g)$ be a smooth Riemannian manifold. An integral two-varifold
is represented by a countably two-rectifiable set $S$ and an integer
multiplicity $\vartheta\ge1$:
\[
 V(\phi)=\int_S\phi(x,T_xS)\vartheta(x)\,d\mathcal H_g^2(x).
\]
We write $\nu=\|V\|_g$ for its weight measure and
$\Mass_g(V)=\nu(M)$ for its mass. For a compactly supported smooth
vector field $X$, the first variation is
\[
 \delta_gV(X)=\int\diver_P^gX\,dV(x,P).
\]
If $\delta_gV(X)=-\int\langle H,X\rangle_g\,d\nu$ for all such
$X$, with $|H|\le\Lambda$ almost everywhere with respect to $\nu$,
then $H$ is the generalized mean curvature of $V$. In this case,
the first variation has no singular part with respect to $\nu$.
The estimate $|\delta_gV(X)|\le\Lambda\int|X|\,d\nu$ for every
such $X$ implies this representation.

\begin{lemma}\label{x:gmt-input}
Fix a smooth ambient Riemannian manifold. The following statements hold.
\begin{enumerate}[label=(\roman*)]
\item Suppose that a sequence of integral varifolds has locally
uniformly bounded mass and total variation of first variation. Then a
subsequence converges weakly to an integral varifold.
\item Suppose that $V_j\rightharpoonup V$ and that $|\delta_gV_j(X)|\le\Lambda\int|X|\,d\|V_j\|_g$ holds for every test
field $X$, with $\Lambda$ independent of $j$. Then the same
inequality holds for $V$. For each fixed $X$, this follows
from weak convergence on the Grassmann bundle.
\item Suppose that an integral varifold has bounded generalized mean
curvature and no singular part of the first variation. Then its
two-dimensional density exists and is upper semicontinuous. The density
is at least one at every point of the support. Indeed, integer
multiplicity gives this bound almost everywhere. Upper semicontinuity
then gives it at every point of the support.
\item At each such support point, a subsequence of the rescalings
converges to a nonzero stationary integral cone. In coordinates for which
the metric at the point is Euclidean, every tangent cone $C$
satisfies
\[
 \|C\|(B_\rho(0))=\pi\theta\rho^2,
 \qquad z\in P\quad C\text{-almost everywhere},
\]
where $\theta$ is the density at that point. The tangent cone need not
be unique or a plane.
\end{enumerate}
\end{lemma}

In (iv), we rescale the varifold by the map $z\mapsto z/r$.
Its two-dimensional Jacobian is $r^{-2}$. The relation $z\in P$
follows from the equality case of the monotonicity formula.
We use these two properties to compute the integral involving the
cutoff function in Section~\ref{sec:singular-flux}. For a smooth
Riemannian metric, the corresponding statements follow in normal
coordinates.

We also use the following consequence of rectifiability. If $E$ is
a smooth hypersurface, then $T_xS\subset T_xE$ at almost every
point of $S\cap E$ with respect to two-dimensional area.
Thus, on the part of the varifold supported on $E$, the first
variation depends only on derivatives in directions tangent to $E$.

We will apply these results with a fixed smooth round metric on $S^4$.
For nonzero varifolds with a uniform bound on mean curvature and
density at least one, the monotonicity formula gives a uniform
positive lower bound for the mass in a ball of fixed small radius
centered at a support point. Since $S^4$ is compact, total mass
converges under weak convergence. Hence a weak limit of such
varifolds is nonzero.

\subsection{The doubled ball and Brendle's vector field}\label{x:why-double}
Fix $R>0$, and write
\[
 B_R=\{x\in\R^4:|x|<R\}=B^4(\pi R^2).
\]
The compactness theorem of Pigati--Stern applies on closed
manifolds. In Section~\ref{c:sec:closed-AH}, we will reflect
AH configurations on $B_R$ to the double of $\overline{B_R}$.
We now choose smooth metrics on this double and prove
the mass estimate needed for that argument.

Let $h_0$ be the standard metric on
\[
 S^4=\{(\xi,t)\in\R^4\times\R:|\xi|^2+t^2=1\},
 \qquad E=\{t=0\}.
\]
Stereographic projection from the opposite pole, followed by
multiplication by $R$, identifies each closed hemisphere with
the closed Euclidean ball $\overline{B_R}$.
On the upper hemisphere, the map is $x=R\xi/(1+t)$.
The Euclidean metrics on the two copies define the metric
\begin{equation}\label{c:eq:double-metric}
 g_{\mathrm d}=f^2h_0,\qquad f(t)=\frac{R}{1+|t|}.
\end{equation}
This metric is Lipschitz and is smooth away from $E$.
For $\eta>0$, we define
\begin{equation}\label{c:eq:smoothed-double}
 f_\eta(t)=\frac{R}{1+\sqrt{t^2+\eta^2}},\qquad
 h_\eta=f_\eta^2h_0.
\end{equation}
Each $h_\eta$ is smooth and invariant under reflection. We have
$f_\eta\to f$ uniformly and $f_\eta\le f$. For $0<\eta\le1$,
the functions $f_\eta$ are bounded below by a positive constant,
and their first derivatives are bounded independently of $\eta$.

We set
\begin{equation}\label{c:eq:m-eta}
 m_\eta=\inf\{\Mass_{h_\eta}(W):W\ne0\text{ is a stationary integral two-varifold on }S^4\}.
\end{equation}
An empty infimum is $+\infty$.
\begin{lemma}\label{c:prop:double-gap}
For the metrics \eqref{c:eq:smoothed-double}, we have
\begin{equation}\label{c:eq:geometric-gap}
 \liminf_{\eta\downarrow0}m_\eta\ge2\pi R^2.
\end{equation}
\end{lemma}
In order to prove Lemma~\ref{c:prop:double-gap}, we use the vector
field in \cite[Section 2, Lemmas 6--8]{Brendle2012}. The corresponding
area estimate for smooth surfaces with free boundary was also proved
in \cite{FraserSchoen}. We will apply the first variation identity to
this field after multiplying it by a cutoff function near its
singular point. The argument allows positive mass on $E$ and does
not require the varifolds to be invariant under reflection.

Fix $y\in\partial B_R$. On each copy of the ball, we define
\begin{equation}\label{c:eq:Brendle-field}
 X_y(x)=\frac{x}{2R^2}-\frac{x-y}{|x-y|^2}.
\end{equation}
If $|x|=|y|=R$, then $(x-y)\cdot x=|x-y|^2/2$.
Hence $X_y\cdot x=0$, and $X_y$ is tangent to the boundary
away from $y$. For every two-plane $P$,
\begin{equation}\label{c:eq:Brendle-div}
 \diver_P X_y=\frac1{R^2}-2\frac{|(x-y)^{\perp_P}|^2}{|x-y|^4}
 \le\frac1{R^2}.
\end{equation}
To see this, put $q=x-y$. Then
\[
 D_e(q/|q|^2)=\frac{e}{|q|^2}-2\frac{(q\cdot e)q}{|q|^4},
\]
and summing over an orthonormal basis of $P$ gives
\eqref{c:eq:Brendle-div}.

We use the same field on both copies of the ball. In signed collar
coordinates, we reflect the tangential components evenly and the
normal component oddly. The normal component vanishes on
$E\setminus\{y\}$, so the values from the two hemispheres agree
there. We denote the field on the double by $X_y$ as well.
It is continuous and locally Lipschitz away from $y$, smooth up to
$E\setminus\{y\}$ on each hemisphere, and tangent to $E$.

For an integral two-varifold, rectifiability gives $P\subset T_xE$
at almost every $(x,P)$ lying over $E$. The restriction of $X_y$
to $E\setminus\{y\}$ is smooth. Its divergence along these planes
is therefore defined by tangential derivatives and satisfies
\eqref{c:eq:Brendle-div}. Thus the same bound applies to the part
of the mass supported on $E$.

To compute the integral near $y$, choose smooth coordinates $z$
centered at $y$ such that $g_{\mathrm d}(y)=\Id$ and
$E=\{z_4=0\}$. These are obtained by multiplying normal coordinates
for $h_0$ by a constant. The differentials of the hemisphere charts
are isometries at $y$. Taylor expansion on each hemisphere gives
\begin{equation}\label{c:eq:singular-field}
 X_y(z)=-\frac{z}{|z|^2}+O(1),\qquad
 g_{\mathrm d}(z)=\Id+O(|z|).
\end{equation}
Here $O(1)$ bounds only the value of the remainder in the vector
field.

\subsection{Proof of the mass estimate}\label{sec:singular-flux}
\paragraph*{Compactness.}
We begin the proof of Lemma~\ref{c:prop:double-gap}.
Suppose that \eqref{c:eq:geometric-gap} does not hold.
Then there exist $\eta_j\downarrow0$, $\delta>0$, and nonzero $h_{\eta_j}$-stationary
integral varifolds $W_j$ such that
\begin{equation}\label{c:eq:contradiction-mass}
 \Mass_{h_{\eta_j}}(W_j)\le2\pi R^2-\delta.
\end{equation}
In order to apply Allard's compactness theorem with a fixed metric,
we express the varifolds with respect to $h_0$.
Let $V_j$ have the same rectifiable set and integer multiplicity as
$W_j$, with area measured by $h_0$.
Put $f_j=f_{\eta_j}$ and $\mu_j=\|V_j\|_{h_0}$. Then
\[
 \Mass_{h_{\eta_j}}(W_j)=\int f_j^2\,d\mu_j.
\]
We next express stationarity in terms of $h_0$.
For a smooth vector field $X$, differentiating this weighted area
along its flow and using the stationarity of $W_j$ gives
\begin{equation}\label{c:eq:conformal-stationarity}
 0=\int\left(f_j^2\diver_P^{h_0}X+2f_jX(f_j)\right)dV_j(x,P).
\end{equation}
We substitute $X=f_j^{-2}Y$. The tangential derivatives of $f_j$ cancel,
and we obtain
\begin{equation}\label{c:eq:fixed-H}
 \delta_{h_0}V_j(Y)
 =-\int\left\langle2(\nabla^{h_0}\log f_j)^{\perp_P},Y\right\rangle_{h_0}\,dV_j.
\end{equation}
Moreover,
\[
 \nabla^{h_0}\log f_\eta
 =-\frac{t\nabla^{h_0}t}
 {\sqrt{t^2+\eta^2}(1+\sqrt{t^2+\eta^2})},
 \qquad |\nabla^{h_0}t|\le1.
\]
Hence $|H_j^{h_0}|\le2$ and $|\delta_{h_0}V_j|\le2\mu_j$.

The positive lower bound for $f_j$ and
\eqref{c:eq:contradiction-mass} give a uniform bound for
$\mu_j(S^4)$. We apply Lemma~\ref{x:gmt-input}(i) and take a
subsequence converging weakly to an integral varifold $V$.
For each smooth test field $Y$, we pass to the limit in the
inequality obtained from \eqref{c:eq:fixed-H}. This gives
\begin{equation}\label{c:eq:limit-H-bound}
 |\delta_{h_0}V(Y)|\le2\int|Y|\,d\mu,\qquad \mu=\|V\|_{h_0}.
\end{equation}
Therefore $V$ has bounded generalized mean curvature, and its
first variation has no additional singular measure.

Next, we show that $V\ne0$. For each $V_j$, take a point $x_j$
whose density is at least one. We embed the round sphere in $\R^5$.
The Euclidean mean curvature is still bounded independently of
$j$. By the monotonicity formula, there exist constants $r_0,c_0>0$,
independent of $j$, such that
\[
 \mu_j(B_{r_0}(x_j))\ge c_0.
\]
Since $S^4$ is compact, we have
$\mu(S^4)=\lim_j\mu_j(S^4)\ge c_0$. Thus $V\ne0$.

By Lemma~\ref{x:gmt-input}(iii)--(iv), the density of $V$ exists
and is at least one at every point of its support. Its tangent
varifolds are stationary integral cones.

We now pass the mass bound to the limit. The uniform convergence
$f_j\to f$ and the weak convergence of $\mu_j$ give
\begin{equation}\label{c:eq:limiting-mass}
 \mu_{\mathrm d}(S^4):=\int f^2\,d\mu
 =\lim_j\int f_j^2\,d\mu_j\le2\pi R^2-\delta.
\end{equation}
Indeed, the difference is at most
\[
 \|f_j^2-f^2\|_\infty\mu_j(S^4)
 +\left|\int f^2\,d\mu_j-\int f^2\,d\mu\right|,
\]
which tends to zero. We write $d\mu_{\mathrm d}=f^2d\mu$.

\paragraph*{First variation at the equator.}
In order to use $X_y$ with a cutoff function near $y$, we need the
first variation identity for vector fields that are Lipschitz on
$S^4$ and smooth up to $E$ on each hemisphere.
We first prove the identity for smooth vector fields tangent to $E$.
Let $X$ be such a field. Since $X(t)=0$ on $E$, in a collar of $E$
we can write
\[
 X(t)=t\,a(q,t)
\]
with $a$ smooth. Hence
\begin{equation}\label{c:eq:tangent-derivative}
 X(f_\eta)
 =-\frac{Rt^2a(q,t)}{\sqrt{t^2+\eta^2}(1+\sqrt{t^2+\eta^2})^2}
 \longrightarrow-\frac{R|t|a(q,t)}{(1+|t|)^2}
\end{equation}
uniformly. We denote the limit by $X(f)$ and set $X(f)=0$ on $E$.
For uniform convergence, put $r=|t|$ and $s=\sqrt{r^2+\eta^2}$.
Then $0\le s-r\le\eta$ and $0\le r-r^2/s\le\eta$.
The function $s\mapsto(1+s)^{-2}$ is uniformly Lipschitz for
$s\ge0$.

Thus the integrands in \eqref{c:eq:conformal-stationarity}
converge uniformly on the Grassmann bundle to a continuous function.
Weak convergence then gives
\begin{equation}\label{c:eq:limit-tangent-stationarity}
 \int\left(f^2\diver_P^{h_0}X+2fX(f)\right)dV=0
 \qquad (X\text{ smooth and tangent to }E).
\end{equation}
On each open hemisphere, this is the usual Euclidean stationarity
by \eqref{c:eq:double-metric}. Across $E$, we use the identity
only for fields tangent to $E$.
The following approximation extends it to the fields needed below.

\begin{lemma}\label{c:lem:piecewise-tests}
Let $X$ be a Lipschitz vector field on $S^4$, tangent to $E$
and smooth up to $E$ on each side.
Then \eqref{c:eq:limit-tangent-stationarity} holds for $X$.
\end{lemma}

\begin{proof}[\bf Proof of Lemma~\ref{c:lem:piecewise-tests}]
We smooth the coefficients of $X$ by convolution in charts near $E$.
Let $Y_j$ be a local approximation in a collar chart $(q,t)$.
In order to keep the field tangent to $E$, we replace its normal
component by
\[
 X_j^t(q,t)=Y_j^t(q,t)-\beta(t)Y_j^t(q,0),
\]
where $\beta$ is a smooth cutoff equal to one near $t=0$.
A partition of unity gives smooth vector fields $X_j$ on $S^4$
which are tangent to $E$.

We apply \eqref{c:eq:limit-tangent-stationarity} to $X_j$.
For each fixed $\rho>0$, the fields converge to $X$ in $C^1$
on $\{|t|\ge\rho\}$, so the integrals over this region converge.

On $E$, rectifiability gives $P\subset T_xE$ for $V$-almost
every $(x,P)$ lying over $E$. Hence the divergence involves only
tangential derivatives. These converge to the derivatives of the
smooth restriction $X|_E$. Since $X_j(f)=X(f)=0$ on $E$, the
integrals over $E$ also converge.

Since $X$ is Lipschitz, the first derivatives of the approximating
fields and of the correction above are uniformly bounded.
Thus the absolute value of each integral over
$\{0<|t|<\rho\}$ is at most
\[
 C\mu(\{0<|t|<\rho\}),
\]
with $C$ independent of $j$ and $\rho$.
This bound tends to zero as $\rho\downarrow0$.
Letting $j\to\infty$ and then $\rho\downarrow0$ proves
\eqref{c:eq:limit-tangent-stationarity}.
\end{proof}

For these fields, we define $\diver_P^{g_{\mathrm d}}X$ by the
usual formula on each open hemisphere. On $E$, the planes are
contained in $TE$ almost everywhere, so we use the tangential
derivatives of the smooth restriction $X|_E$.
The conformal formula and \eqref{c:eq:limit-tangent-stationarity}
give
\begin{equation}\label{c:eq:piecewise-stationarity}
 \int\diver_P^{g_{\mathrm d}}X\,d\mu_{\mathrm d}=0.
\end{equation}
In order to choose the singular point $y$ in the support of $V$,
we show that this support meets $E$. Otherwise, the
support is compact and has positive distance from $E$. Its
restriction to each hemisphere is then a compactly supported stationary
varifold in a Euclidean ball. We take a cutoff of the radial field
$x$ which is equal to $x$ near this support. The tangential
divergence is two there. By stationarity, the mass on that hemisphere
is zero. Applying the same argument on the other hemisphere contradicts
$V\ne0$. Therefore
\begin{equation}\label{c:eq:support-meets-E}
 \spt V\cap E\ne\varnothing.
\end{equation}
The intersection $\spt V\cap E$ may have zero two-dimensional
measure. We use only the existence of a support point on $E$.

\paragraph*{The mass estimate.}
\begin{proof}[\bf Completion of the proof of Lemma~\ref{c:prop:double-gap}]
Take $y\in\spt V\cap E$ and use the coordinates in
\eqref{c:eq:singular-field}. Let $\theta\ge1$ be the density at
$y$, computed with the metric $g_{\mathrm d}$ at that point. Then
\[
 \mu_{\mathrm d}(\{|z|<r\})=\pi\theta r^2+o(r^2).
\]
The constant change of scale at $y$ preserves this multiplicity
density, so $\theta$ is also the density computed with the smooth
metric $h_0$.

Let $\chi:[0,\infty)\to[0,1]$ be smooth and nondecreasing,
equal to zero on $[0,1]$ and one on $[2,\infty)$.
Put $\chi_s(z)=\chi(|z|/s)$ near $y$ and extend it by one outside
the chart. The field $\chi_sX_y$ satisfies the hypotheses of
Lemma~\ref{c:lem:piecewise-tests}. Applying that lemma gives
\begin{equation}\label{c:eq:flux-balance}
 0=I_s+J_s,
 \quad I_s=\int\chi_s\diver_P^{g_{\mathrm d}}X_y\,d\mu_{\mathrm d},
 \quad J_s=\int\langle\nabla_P^{g_{\mathrm d}}\chi_s,X_y\rangle_{g_{\mathrm d}}\,d\mu_{\mathrm d}.
\end{equation}
The divergence bound \eqref{c:eq:Brendle-div} holds on each
hemisphere and on the tangent planes contained in $TE$. Hence
\begin{equation}\label{c:eq:I-bound}
 I_s\le R^{-2}\mu_{\mathrm d}(S^4).
\end{equation}

It remains to compute $J_s$, the term containing the derivative
of $\chi_s$. We will rescale near $y$ and pass to a tangent cone.
First, on the annulus $s<r=|z|<2s$, we measure the same rectifiable
set and multiplicity with Euclidean area in the coordinates $z$.
Denote the area measure and varifold by $\nu$ and $\mathcal V$,
respectively. Let $\pi_P$ be Euclidean orthogonal projection onto
the approximate tangent plane. By \eqref{c:eq:singular-field},
\begin{equation}\label{c:eq:J-main}
 J_s=-\int_{s<r<2s}\frac{\chi'(r/s)}{sr}
            \frac{|\pi_Pz|^2}{r^2}\,d\nu(z)+O(s).
\end{equation}
The bounded remainder in $X_y$ contributes at most
$Cs^{-1}\mu_{\mathrm d}(B_{2s})=O(s)$.
On this annulus, the errors in the metric, orthogonal projection,
and area element are $O(s)$. The leading integrand is $O(s^{-2})$
and the mass of the annulus is $O(s^2)$, so these errors also
contribute $O(s)$.

Put $\mathcal V_s=(z\mapsto z/s)_\#\mathcal V$.
The two-dimensional Jacobian of this map is $s^{-2}$, so
\eqref{c:eq:J-main} becomes
\[
 J_s=-\int\frac{\chi'(|\xi|)}{|\xi|}
       \frac{|\pi_P\xi|^2}{|\xi|^2}\,d\mathcal V_s(\xi,P)+O(s).
\]
Choose $s_j\downarrow0$ such that $\mathcal V_{s_j}$ converges
to a stationary integral tangent cone $C$.
The integrand is continuous and supported in a fixed annulus away
from zero, so we can pass to the limit by weak convergence.
By Lemma~\ref{x:gmt-input}(iv), we have $\xi\in P$ almost
everywhere on $C$, and
\[
 \|C\|(B_\rho)=\pi\theta\rho^2.
\]
It follows that
\begin{align}
 \lim_jJ_{s_j}
 &=-\int\frac{\chi'(|\xi|)}{|\xi|}\,d\|C\|(\xi)\notag\\
 &=-\int_0^\infty\frac{\chi'(\rho)}{\rho}\,d(\pi\theta\rho^2)
 =-2\pi\theta.\label{c:eq:flux-limit}
\end{align}
Here $|\pi_P\xi|^2/|\xi|^2=1$ on the tangent cone because
$\xi\in P$.

For each $s$, equations \eqref{c:eq:flux-balance}--\eqref{c:eq:I-bound} imply
\[
 R^{-2}\mu_{\mathrm d}(S^4)\ge -J_s.
\]
We let $s=s_j\downarrow0$ and obtain
\[
 \mu_{\mathrm d}(S^4)\ge2\pi R^2\theta\ge2\pi R^2,
\]
which contradicts \eqref{c:eq:limiting-mass} and proves
\eqref{c:eq:geometric-gap}.
\end{proof}

\section{The first min--max value on the ball}
\label{c:sec:closed-AH}\label{c:sec:normalization}
The purpose of this section is to use the geometric estimate of
Section~\ref{c:sec:double} to prove the lower bound on the Euclidean
ball. We first construct a nowhere-zero section of the evaluation line
on a low AH sublevel on a fixed smooth closed manifold. For this, we
use the absence of positive critical values below the geometric
threshold and an equivariant deformation to a sufficiently small
sublevel. We then reflect configurations from the ball to its smooth
double and correct the gauge transformation law of the resulting
section. This gives a trivialization of the evaluation line on the
original ball sublevel, which excludes all sweepouts below the required
energy. The upper bound for the cylinder  then gives the first min–max value on the standard ball.

We use the closed-manifold AH results in \cite{PigatiStern2021}, including the
corrected Finsler Palais--Smale statement. See also \cite[Remark 5.1]{ParisePigatiStern2024}. For the
deformation, we use the pseudogradient and averaging constructions in
\cite[Theorems 2.2.1 and 2.5.5--2.5.6]{StruweNotes}. We give the construction of the vector field because we
need one deformation defined on the whole sublevel.

\subsection{AH theory on closed manifolds}\label{sec:closed-AH-input}
Fix a smooth metric $g$ on $Y\cong S^4$.

We choose a smooth radial potential $\widetilde W$ which agrees with
\[
 W_0(u)=\frac{(1-|u|^2)^2}{4}
\]
for $|u|\le3/2$, is at least $W_0$, has positive radial derivative for
$|u|>1$, and equals $|u|^p$ for $|u|\ge2$, with a fixed $p>4$.
 We denote by $\Et_{\eps,g}$ the energy
\eqref{c:eq:energy} with $\widetilde W/\eps^2$ instead of $W_0/\eps^2$.

We use the complete Coulomb Banach space
\begin{equation}\label{c:eq:Banach-space}
 \mathcal X_g=\{(u,A):u\in W^{1,2}(Y,\C)\cap L^p(Y),\;
 A\in W^{1,2}(Y,T^*Y),\;d_g^*A=0\},
\end{equation}
with norm $\|u\|_p+\|du\|_2+\|A\|_{W^{1,2}}$. We regard this as a real Banach space. Since $H^1(Y;\R)=0$, the residual gauge group
consists of constant circle-valued transformations.

We use the following results from \cite{PigatiStern2021}. The proposition numbers
are  those of the corrected authors' version.
\begin{enumerate}[label=(\roman*)]
\item The modified functional is $C^1$. Criticality on the
Coulomb space is equivalent to criticality before gauge fixing. The
critical points are smooth up to gauge (\cite[Proposition~7.5]{PigatiStern2021}).
\item The Palais--Smale condition holds for the gauge-invariant
Finsler norm on the Coulomb space ( \cite[Proposition~7.6]{PigatiStern2021}). See also \cite[Remark~5.1]{ParisePigatiStern2024}.
\item For a fixed smooth metric and sufficiently small $\eps$,
there is a uniform positive lower bound for the positive critical
energies (\cite[Proposition~7.11]{PigatiStern2021}).
\item For a sequence of critical points with bounded energy and
scales tending to zero, the normalized energy measures converge along
a subsequence to the weight measure of a stationary integral
two-varifold (\cite[Theorem~1.1]{PigatiStern2021}).
\end{enumerate}
The last two results also apply to the modified functional. Indeed,
by the maximum principle its critical points satisfy $|u|\le1$. For
example, with $\Delta=\diver\nabla$ and $\widetilde W(u)=\widetilde w(|u|)$, the section equation implies
\[
 \Delta|u|^2=2|d_Au|^2+\eps^{-2}|u|\widetilde w'(|u|).
\]
This excludes a maximum above one. Hence these critical points satisfy
the original equations.

The Finsler norm at $(u,A)$ is
\[
 \|(v,B)\|_{F,(u,A)}=\|v\|_p+\|d_Av\|_2+\|B\|_{W^{1,2}}.
\]
On a set of bounded energy, with $\eps$ and $g$ fixed, Hodge
theory gives
\begin{equation}\label{c:eq:A-bound-fixed}
 \|A\|_{W^{1,2}}\le C_g\|dA\|_2
 \le C_g\eps^{-1}\Et_{\eps,g}(u,A)^{1/2}.
\end{equation}
By the four-dimensional Sobolev inequality, we also have $\|Av\|_2\le C_g\|A\|_{W^{1,2}}\|v\|_p$.
Thus the Finsler norm and the standard norm are uniformly equivalent
on this set. On $Y\cong S^4$, harmonic circle-valued gauges are constant.
We can therefore choose a convergent subsequence of the remaining
constant gauges. The Palais--Smale statement then gives strong
convergence in \eqref{c:eq:Banach-space} on each fixed positive energy band. The
constants in the norm comparison may depend on $\eps$. We do not
need these constants to be uniform as $\eps\downarrow0$.

\subsection{The evaluation line on low energy sublevels}\label{sec:Z0-construction}
We now construct the deformation for a finite time. First, we choose
local directions in which the energy decreases. We combine them by a
locally Lipschitz partition of unity and then average over the constant
circle. This is the construction in \cite[Theorems 2.2.1 and 2.5.5--2.5.6]{StruweNotes}.

\begin{lemma}\label{lem:equivariant-descent}
Fix $g$ and $\eps$. Write $F=\Et_{\eps,g}$ on the complete real Banach
space $\mathcal X_g$ in \eqref{c:eq:Banach-space}. Let $0<\delta<L<L'$ be finite. Suppose that
$F$ has no critical point with $\delta/4\le F\le L'$ and satisfies the
Palais--Smale condition on this band in the standard Banach norm.
Then there exists a bounded locally Lipschitz field $Z:\mathcal X_g\to\mathcal X_g$ whose
flow $\Phi_t$ satisfies
\begin{equation}\label{eq:descent-map-properties}
 \Phi_t(R_\theta\xi)=R_\theta\Phi_t(\xi),\qquad
 F(\Phi_t(\xi))\le F(\xi)\quad(t\ge0),\qquad
 \Phi_T(\{F<L\})\subset\{F<\delta\}
\end{equation}
for a finite time $T$. Here $R_\theta(u,A)=(e^{i\theta}u,A)$. The flow depends continuously on $(t,\xi)$.
There exists a compact interval $I\subset(\delta/4,L')$ such that $Z(\xi)=0$ whenever
$F(\xi)\notin I$.
\end{lemma}
\begin{proof}[\bf Proof of Lemma~\ref{lem:equivariant-descent} ]
If $\{F<L\}\subset\{F<\delta/2\}$, we take $Z=0$. Suppose otherwise. Then the energy band
is nonempty. By the Palais--Smale condition and the absence of critical 
points, we have
\begin{equation}\label{eq:descent-beta}
 \beta:=\inf_{\delta/4\le F(\xi)\le L'}
       \|DF(\xi)\|_{\mathcal X_g^*}>0.
\end{equation}
Indeed, if this infimum were zero, there would be a sequence in the band
such that $DF\to0$. The Palais--Smale condition would give a strongly
convergent subsequence. Its limit would be a critical point in the same
band. This is a contradiction.

\emph{Step 1.}
Let
\[
 \mathcal O=\{\xi\in\mathcal X_g:\delta/4<F(\xi)<L'\}.
\]
For each $\xi_0\in\mathcal O$, we use the definition of the dual norm and choose
a fixed vector $v_{\xi_0}\in\mathcal X_g$ of norm at most one such that
\[
 DF(\xi_0)[v_{\xi_0}]<-\tfrac12\|DF(\xi_0)\|_{\mathcal X_g^*}.
\]
Since $DF$ is continuous, there is a neighborhood $U_{\xi_0}$ on
which this same vector satisfies
\[
 DF(\xi)[v_{\xi_0}]
 \le-\tfrac14\|DF(\xi)\|_{\mathcal X_g^*}.
\]
We take a locally finite refinement $\{U_i\}$ of these neighborhoods
and a subordinate locally Lipschitz partition of unity $\{\lambda_i\}$ on
$\mathcal O$. Then $\lambda_i\ge0$, $\sum_i\lambda_i=1$, and only finitely many terms are
nonzero near each point. Let $v_i$ be the fixed descent vector for
the neighborhood containing $U_i$. We define
\begin{equation}\label{eq:Z0-definition}
 Z_0:\mathcal O\longrightarrow\mathcal X_g,\qquad
 Z_0(\xi)=\sum_i\lambda_i(\xi)v_i.
\end{equation}
Here $Z_0(\xi)$ is a tangent vector to the configuration space $\mathcal X_g$.

The sum is finite near each point, so $Z_0$ is locally Lipschitz.
Since the coefficients give a convex combination, $\|Z_0\|\le1$. By the
linearity of $DF(\xi)$, we have
\begin{equation}\label{eq:Z0-bounds}
 \|Z_0(\xi)\|_{\mathcal X_g}\le1,\qquad
 DF(\xi)[Z_0(\xi)]
   =\sum_i\lambda_i(\xi)DF(\xi)[v_i]
   \le-\tfrac14\|DF(\xi)\|_{\mathcal X_g^*}
   \le-\beta/4.
\end{equation}

\emph{Step 2.}
The circle acts linearly and isometrically on $\mathcal X_g$, and $F$
is invariant under this action. Therefore $\mathcal O$ is also invariant.
Put
\begin{equation}\label{eq:Z-average}
 Z_{\mathrm{av}}(\xi)=\frac1{2\pi}\int_0^{2\pi}
       R_{-\theta}Z_0(R_\theta\xi)\,d\theta.
\end{equation}
The integrand is continuous in $\theta$, so this is defined.  Since $F\circ R_\theta=F$,
\[
 DF(\xi)[R_{-\theta}w]=DF(R_\theta\xi)[w].
\]
It follows that
\begin{equation}\label{eq:Z-average-descent}
 DF(\xi)[Z_{\mathrm{av}}(\xi)]\le-\beta/4,\qquad
 \|Z_{\mathrm{av}}(\xi)\|\le1,\qquad
 Z_{\mathrm{av}}(R_\theta\xi)=R_\theta Z_{\mathrm{av}}(\xi).
\end{equation}
Note that this field is also locally Lipschitz.

\emph{Step 3.}
We choose $\kappa\in C_c^\infty((\delta/4,L');[0,1])$ which is equal to one on $[\delta/2,L]$, and define
\begin{equation}\label{eq:Z-final}
 Z(\xi)=\kappa(F(\xi))Z_{\mathrm{av}}(\xi)
       \quad(\xi\in\mathcal O),
 \qquad Z(\xi)=0\quad(\xi\notin\mathcal O).
\end{equation}
By the support condition on $\kappa$, this extends as a locally
Lipschitz field to all of $\mathcal X_g$. Its norm is at most one, and it is
equivariant. We apply local existence and uniqueness for ODEs on a
Banach space to
\begin{equation}\label{eq:configuration-ODE}
 \dot\xi(t)=Z(\xi(t)),\qquad \xi(0)=\xi_0.
\end{equation}
Since $\|Z\|\le1$, we have $\|\xi(t)-\xi(s)\|\le|t-s|$. In addition, the flow exists for every finite time.
Local Lipschitz continuity gives uniqueness and continuous dependence on the initial condition. Let $\Phi_t: \mathcal X_g \to \mathcal X_g$ denote the flow solving (\ref{eq:configuration-ODE}).

\emph{Step 4. }
Along each trajectory, $dF/dt=DF[Z]\le0$. While $\delta/2\le F\le L$, this derivative is
at most $-\nu$, where $\nu=\beta/4$. We take $T=2L/\nu$. Suppose that a
trajectory with $F(\xi_0)<L$ remains above $\delta/2$ until $T$. Then
$F(\Phi_T(\xi_0))\le F(\xi_0)-\nu T<0$, which contradicts $F\ge0$. Therefore the trajectory reaches
$\{F\le\delta/2\}$ before $T$. Since $F$ is nonincreasing along the
flow, we obtain $F(\Phi_T(\xi_0))\le\delta/2<\delta$. This proves \eqref{eq:descent-map-properties}.
\end{proof}

We use this deformation to trivialize the evaluation line on a low
sublevel. We set
\[
 m(g)=\inf\{\Mass_g(V):V\ne0\text{ is a stationary integral two-varifold on }(Y,g)\}.
\]
We take the empty infimum to be infinity.

\begin{lemma}\label{c:prop:closed-triviality}
Let $Y\cong S^4$ have a fixed smooth metric $g$, and let $0<L<m(g)$ be
finite. For all sufficiently small $\eps>0$, there exists a continuous
map
\begin{equation}\label{c:eq:Xi-statement}
 \Xi_{\eps,g,L}:\{\Et_{\eps,g}<L\}\subset\mathcal X_g\longrightarrow S^1,
 \qquad
 \Xi_{\eps,g,L}(e^{i\theta}u,A)=e^{i\theta}\Xi_{\eps,g,L}(u,A).
\end{equation}
In particular, the associated line bundle over the quotient of this sublevel by the constant circle is trivial.

\end{lemma}
\begin{proof}[\bf Proof of Lemma \ref{c:prop:closed-triviality}]
Choose $L<L'<m(g)$. First, we show that there are no positive critical
values at most $L'$ for all sufficiently small $\eps$. Suppose
otherwise. Then there are $\eps_j\downarrow0$ and critical points with energies
in $(0,L']$. By (iii), these energies have a positive lower bound.
We take a subsequence whose energies converge to $\ell\in(0,L']$. By (iv),
the normalized energy measures converge to $\|V\|_g$ for a stationary
integral two-varifold $V$. Since $Y$ is compact, we can use
the constant function one to test this convergence. We obtain
\[
 \Mass_g(V)=\ell>0.
\]
This contradicts $\ell\le L'<m(g)$ and the definition of $m(g)$.

Fix such an $\eps$. Next, we show that the average of $u$ is
nonzero on a sufficiently small energy sublevel. If $\Et_{\eps,g}(u,A)\le\delta\le1$ and
$\eps\le1$, Hodge theory and the potential estimate give
\begin{equation}\label{c:eq:tiny-estimates}
 \|A\|_{W^{1,2}}\le C_g\eps^{-1}\sqrt\delta,
 \quad\|u\|_p\le C_g,
 \quad\|1-|u|\|_2\le C_g\eps\sqrt\delta.
\end{equation}
Therefore
\[
 \|du\|_2\le\|d_Au\|_2+\|Au\|_2
 \le C_g(1+\eps^{-1})\sqrt\delta.
\]
Put $u_{\av}=\Vol_g(Y)^{-1}\int_Yu\,dV_g$. By Poincar\'e's inequality,
\begin{align}
 \Vol_g(Y)^{1/2}\bigl|1-|u_{\av}|\bigr|
 &\le\|1-|u|\|_2+\|u-u_{\av}\|_2\notag\\
 &\le C_g(\eps+1+\eps^{-1})\sqrt\delta.\label{c:eq:average-bound}
\end{align}
We choose $0<\delta_\eps<\min\{1,L\}$ such that $|u_{\av}|>1/2$ whenever $\Et_{\eps,g}<\delta_\eps$. Then
\begin{equation}\label{c:eq:average-phase}
 \Phi_g(u,A)=\frac{\int_Yu\,dV_g}{|\int_Yu\,dV_g|}
\end{equation}
is a continuous $S^1$-equivariant map on this sublevel. The integral
is defined for a Sobolev section.

We apply Lemma~\ref{lem:equivariant-descent} with $F=\Et_{\eps,g}$ and $\delta=\delta_\eps$. The preceding
argument excludes critical points on $[\delta_\eps/4,L']$. By Section~\ref{sec:closed-AH-input}, the
Palais--Smale condition holds in the standard norm at this fixed scale.
Thus the lemma gives a continuous equivariant time map
\[
 H_{\eps,g,L}=\Phi_T:\{\Et_{\eps,g}<L\}
       \longrightarrow\{\Et_{\eps,g}<\delta_\eps\}.
\]
Its time may depend on $\eps,g,L$.

For sufficiently small $\eps$, the identically zero section has
energy at least
\[
 \frac{\Vol_g(Y)}{8\pi\eps^2}>L'.
\]
The energy does not increase along a trajectory. Hence a trajectory
starting below $L$ never has an identically zero section. The
circle action is therefore free on this sublevel. We define
\begin{equation}\label{c:eq:Xi-formula}
 \Xi_{\eps,g,L}(u,A)=\Phi_g\bigl(H_{\eps,g,L}(u,A)\bigr)
\end{equation}
Then \eqref{c:eq:Xi-statement} holds. Thus $[u,A]\mapsto[(u,A),\Xi_{\eps,g,L}(u,A)]$ is a nowhere-zero section of the
associated line bundle.
\end{proof}

\Needspace{13\baselineskip}
\subsection{Reflection and the lower bound for the ball}\label{sec:ball-reflection}

We apply Lemma~\ref{c:prop:closed-triviality} on a smooth double
of the ball to prove the following lower bound.

\begin{lemma}\label{c:prop:ball-lower}
For every $0<s<\pi R^2$, the associated line bundle on
\[
 \{\En^{B_R}_{\eps,g_0}<s\}\subset\B^*(B_R)
\]
is trivial for all sufficiently small $\eps$. Consequently,
\[
 W_1(B_R,g_0)\ge\pi R^2.
\]
\end{lemma}
\begin{proof}[\bf Proof of Lemma \ref{c:prop:ball-lower}]
Fix $0<s<\pi R^2$ and $x_0\in B_R$.
By Lemma~\ref{c:prop:double-gap}, choose a smooth metric $h_\eta$
such that $m_\eta>2s$, and choose $L$ with $2s<L<m_\eta$.
We fix $\eta$ and $L$ and take $\eps$ sufficiently small for
Lemma~\ref{c:prop:closed-triviality}. Put
\[
\begin{aligned}
 \mathcal U_B
 &=\{(u,A)\in\Conf^*(B_R):\En^{B_R}_{\eps,g_0}(u,A)<s\},\\
 \mathcal U_\eta
 &=\{\Et^{S^4}_{\eps,h_\eta}<L\}\subset\mathcal X_{h_\eta}.
\end{aligned}
\]
To trivialize the line bundle, it suffices to construct a
continuous map $Z_{\eta,\eps,L}:\mathcal U_B\to S^1$ such that
\[
 Z_{\eta,\eps,L}(\gamma\cdot(u,A))
 =\gamma(x_0)Z_{\eta,\eps,L}(u,A).
\]
We obtain this map from $\Xi_{\eps,h_\eta,L}$ by sending
configurations on the ball to $\mathcal U_\eta$.

The map $\Xi_{\eps,h_\eta,L}$ was constructed using the function
$\widetilde W$. The ball min--max values are defined using $W_0$.
We first modify $u$ so that it takes values in the set where
$\widetilde W=W_0$. This modification does not increase the
original energy.
For each $b>1$, choose a smooth nondecreasing function $\tau_b$
satisfying
\[
 \tau_b(r)=r\quad(r\le1),\qquad 0\le\tau_b'\le1,\qquad
 1\le\tau_b(r)\le r\quad(r\ge1),\qquad \tau_b(r)\le b.
\]
Put $T_b(z)=\tau_b(|z|)z/|z|$ and $T_b(0)=0$.
Its radial and angular derivatives have norms at most one, so
$\|DT_b\|\le1$. Since $T_b$ commutes with the circle action,
$DT_b(z)(iz)=iT_b(z)$. Thus, for a tangent vector $v$,
\[
\begin{aligned}
 d_A(T_bu)(v)
 &=DT_b(u)du(v)-iA(v)T_bu\\
 &=DT_b(u)\bigl(du(v)-iA(v)u\bigr).
\end{aligned}
\]
Also, $W_0(T_bu)\le W_0(u)$.
Therefore
\begin{equation}\label{c:eq:truncate}
 d_A(T_bu)=DT_b(u)d_Au,\qquad
 \En_{\eps,g}(T_bu,A)\le\En_{\eps,g}(u,A).
\end{equation}
We fix $T=T_{5/4}$. Since $|Tu|\le5/4<3/2$, the two
potentials agree on its image. Hence
\begin{equation}\label{c:eq:modified-transfer}
 \Et_{\eps,g}(Tu,A)=\En_{\eps,g}(Tu,A)\le\En_{\eps,g}(u,A).
\end{equation}

In order to reflect the connection, we choose a gauge in which
its normal component vanishes on $\partial B_R$.
For a smooth real one-form $A$, let $\phi_A$ solve
\begin{equation}\label{c:eq:ball-gauge}
 d^*(A-d\phi_A)=0,\qquad (A-d\phi_A)(\nu)=0,\qquad
 \int_{B_R}\phi_A\,dx=0.
\end{equation}
Here $\nu$ is the outward unit normal. This is the Neumann problem
$d^*d\phi_A=d^*A$, $\partial_\nu\phi_A=A(\nu)$.
The divergence theorem gives the compatibility condition, and the
mean-zero condition makes the solution unique. We set
\[
 C_B(Tu,A)=(e^{-i\phi_A}Tu,A-d\phi_A).
\]
Every smooth pair admits this gauge choice, and the energy is
unchanged.

We now reflect this pair to the double. Write $(v,B)=C_B(Tu,A)$.
In a collar with $t\ge0$ on the first copy of the ball, write
$B=\sum_{a=1}^3B_a\,dq^a+B_t\,dt$.
We define
\[
\begin{aligned}
 (\mathcal Rv)(q,t)&=v(q,|t|),\\
 (\mathcal RB)(q,t)
 &=\sum_{a=1}^3B_a(q,|t|)\,dq^a
   +\operatorname{sgn}(t)B_t(q,|t|)\,dt.
\end{aligned}
\]
The reflected section and the tangential components of the
reflected connection are continuous across the equator.
The normal component is also continuous because $B_t(q,0)=0$.
Thus $\mathcal Rv\in W^{1,2}\cap L^p$ and
$\mathcal RB\in W^{1,2}$, and their weak derivatives have no
additional term supported on the equator.

On each hemisphere, the metric $h_\eta$ in
\eqref{c:eq:smoothed-double} is $\lambda_\eta^2g_0$ in the ball
chart, where $\lambda_\eta=f_\eta/f\le1$.
In these coordinates, the energy on each hemisphere is
\[
 \frac1{2\pi}\int_{B_R}
 \left(
 \lambda_\eta^2|d_Bv|_{g_0}^2
 +\eps^2|dB|_{g_0}^2
 +\frac{\lambda_\eta^4\widetilde W(v)}{\eps^2}
 \right)dx.
\]
Since $0<\lambda_\eta\le1$ and $\widetilde W\ge0$, this is
at most $\Et^{B_R}_{\eps,g_0}(v,B)$.
We apply this estimate on both hemispheres.
Since $C_B$ preserves the energy, \eqref{c:eq:modified-transfer}
gives
\begin{equation}\label{c:eq:double-energy-two}
 \Et^{S^4}_{\eps,h_\eta}\bigl(\mathcal R C_B(Tu,A)\bigr)
 \le2\En^{B_R}_{\eps,g_0}(u,A).
\end{equation}
The scale $\eps$ is the same on both sides.

To apply $\Xi_{\eps,h_\eta,L}$, we need a configuration in
$\mathcal X_{h_\eta}$. The gauge on the ball was chosen with
respect to $g_0$, so the reflected connection need not satisfy
the Coulomb condition with respect to $h_\eta$.
We therefore choose a Coulomb gauge on $(S^4,h_\eta)$.
We normalize the real function defining this gauge transformation
to have integral zero with respect to $dV_{h_\eta}$.
We denote the  map by $C_\eta$; see
\cite[Remark~7.4]{PigatiStern2021}.

 Define
\begin{equation}\label{c:eq:K-transfer}
 \mathcal K_\eta=C_\eta\mathcal R C_B\circ(T,\Id).
\end{equation}
The Neumann solution depends continuously on $A$, reflection is
bounded on the stated Sobolev spaces, and $C_\eta$ is continuous.
Hence $\mathcal K_\eta$ is a continuous map from smooth ball
configurations to $\mathcal X_{h_\eta}$.
The last gauge transformation preserves energy. Thus, for
$(u,A)\in\mathcal U_B$, \eqref{c:eq:double-energy-two} gives
\[
 \Et^{S^4}_{\eps,h_\eta}\bigl(\mathcal K_\eta(u,A)\bigr)
 \le2\En^{B_R}_{\eps,g_0}(u,A)<2s<L.
\]
Therefore $\mathcal K_\eta(\mathcal U_B)\subset\mathcal U_\eta$,
and $\Xi_{\eps,h_\eta,L}\circ\mathcal K_\eta$ is defined on
$\mathcal U_B$.

We next compute how $\mathcal K_\eta$ changes under a gauge
transformation on the ball. Both $\mathcal R$ and $C_\eta$ commute
with constant gauge transformations. For $C_\eta$, this follows
because its gauge-fixing function depends only on the connection.
Let $\gamma$ be a smooth gauge transformation on $B_R$.
Since the ball is simply connected, we can choose a smooth
real-valued function $\psi$ such that $\gamma=e^{i\psi}$.
This function $\psi$ is a lift of $\gamma$ under the covering map
$\R\to S^1$, $t\mapsto e^{it}$.
Put $\av_{B_R}\psi=\Vol(B_R)^{-1}\int_{B_R}\psi\,dx$.
The equations and boundary conditions in \eqref{c:eq:ball-gauge}
imply that $h=\phi_{A+d\psi}-\phi_A-\psi$ satisfies
\[
 d^*dh=0\quad\text{in }B_R,\qquad
 \partial_\nu h=0\quad\text{on }\partial B_R.
\]
Integration by parts gives $\int_{B_R}|dh|^2\,dx=0$, so $h$ is
constant. Both $\phi_A$ and $\phi_{A+d\psi}$ have integral zero over $B_R$.
Since $h$ is constant, we obtain $h=-\av_{B_R}\psi$. Thus
\begin{equation}\label{c:eq:phi-transform}
 \phi_{A+d\psi}=\phi_A+\psi-\av_{B_R}\psi.
\end{equation}
Since $T(e^{i\psi}u)=e^{i\psi}Tu$, the two components satisfy
\[
\begin{aligned}
 A+d\psi-d\phi_{A+d\psi}&=A-d\phi_A,\\
 e^{-i\phi_{A+d\psi}}T(e^{i\psi}u)
 &=e^{-i(\phi_A+\psi-\av_{B_R}\psi)}e^{i\psi}Tu\\
 &=e^{i\av_{B_R}\psi}e^{-i\phi_A}Tu.
\end{aligned}
\]
Applying $\mathcal R$ and $C_\eta$ therefore gives
\begin{equation}\label{c:eq:K-transform}
 \mathcal K_\eta(\gamma\cdot(u,A))
 =e^{i\av_{B_R}\psi}\mathcal K_\eta(u,A).
\end{equation}
Two lifts differ by a constant in $2\pi\Z$, so this phase is
independent of the lift.

By \eqref{c:eq:K-transform} and the equivariance of
$\Xi_{\eps,h_\eta,L}$, the composite
$\Xi_{\eps,h_\eta,L}\circ\mathcal K_\eta$ is multiplied by
$e^{i\av_{B_R}\psi}$ under $\gamma=e^{i\psi}$.
The evaluation line at $x_0$ uses the character
$\gamma(x_0)=e^{i\psi(x_0)}$.
By \eqref{c:eq:phi-transform}, the factor $e^{i\phi_A(x_0)}$
changes by $e^{i(\psi(x_0)-\av_{B_R}\psi)}$.
We therefore define
\begin{equation}\label{c:eq:Z-map}
 Z_{\eta,\eps,L}(u,A)=e^{i\phi_A(x_0)}
 \Xi_{\eps,h_\eta,L}\bigl(\mathcal K_\eta(u,A)\bigr).
\end{equation}
Then
\begin{equation}\label{c:eq:Z-transform}
\begin{aligned}
 Z_{\eta,\eps,L}(\gamma\cdot(u,A))
 &=e^{i(\phi_A(x_0)+\psi(x_0)-\av_{B_R}\psi)}
   e^{i\av_{B_R}\psi}
   \Xi_{\eps,h_\eta,L}\bigl(\mathcal K_\eta(u,A)\bigr)\\
 &=\gamma(x_0)Z_{\eta,\eps,L}(u,A).
\end{aligned}
\end{equation}
Since $Z_{\eta,\eps,L}$ is continuous and has modulus one,
\[
 [u,A]\longmapsto[(u,A),Z_{\eta,\eps,L}(u,A)]
\]
is a nowhere-zero section of the evaluation line on
$\mathcal U_B/\G(B_R)$. Only the smooth function $\phi_A$ is
evaluated at $x_0$.

This construction also identifies the evaluation line with a
pullback line. By \eqref{c:eq:K-transform}, the map
$\mathcal K_\eta$ induces
\[
 \overline{\mathcal K}_\eta:
 \mathcal U_B/\G(B_R)\longrightarrow\mathcal U_\eta/S^1,
 \qquad [u,A]\longmapsto[\mathcal K_\eta(u,A)].
\]
Let $\LL_\eta=(\mathcal U_\eta\times\C)/S^1$ with the quotient
convention in \eqref{c:eq:eval-line}.
The transformation laws \eqref{c:eq:phi-transform} and
\eqref{c:eq:K-transform} show that the fiber map
\[
 [\mathcal K_\eta(u,A),z]
 \longmapsto[(u,A),e^{i\phi_A(x_0)}z]
\]
is well defined over $[u,A]$ and gives
\[
 \overline{\mathcal K}_\eta^*\LL_\eta
 \cong\LL_{x_0}|_{\mathcal U_B/\G(B_R)}.
\]
Under this isomorphism, the pullback of the section defined by
$\Xi_{\eps,h_\eta,L}$ is the section defined by $Z_{\eta,\eps,L}$.

It follows that $\mu_{B_R}$ vanishes on the quotient sublevel.
If an AH $1$-sweepout $F:K\to\B^*(B_R)$ had maximum energy
less than $s$, its image would lie in this sublevel, so
$F^*\mu_{B_R}=0$. This contradicts the sweepout condition.
Therefore $w_{\eps,1}(B_R,g_0)\ge s$ for all sufficiently small
$\eps$. Taking the lower limit in $\eps$ and then letting
$s\uparrow\pi R^2$, we obtain
\begin{equation}\label{eq:family-ball-lower}
 W_1(B_R,g_0)\ge\pi R^2.
\end{equation}
\end{proof}

Combining this lower bound with the upper bound for the cylinder gives the
first min--max value.

\begin{lemma}\label{c:prop:ball-first}
For every $a>0$, we have
\[
 \lim_{\eps\downarrow0}w_{\eps,1}(B(a),g_0)=a.
\]
\end{lemma}
\begin{proof}[\bf Proof of Lemma \ref{c:prop:ball-first}]
For every $b>a$, we have $B^4(a)\Subset Z^4(b)$.
Lemma~\ref{c:prop:cylinder}, applied with the standard compatible
structure, gives
\[
 \limsup_{\eps\downarrow0}w_{\eps,1}(B^4(a),g_0)\le b.
\]
Let $b\downarrow a$. Lemma~\ref{c:prop:ball-lower} gives the
opposite lower bound, so the limit is $a$.
\end{proof}

\section{Proof of the capacity theorem}\label{c:sec:capacity}
We now prove Theorem~\ref{c:thm:main} for the min--max values \eqref{eq:family-minmax} defined
by finite-complex sweepouts. The obstruction on the ball sublevel
applies to every sweepout. The cylinder families give an upper
bound for every compatible structure.

\begin{proof}[\bf Proof of Theorem~\ref{c:thm:main}]
Let $\varphi:(M,\omega)\hookrightarrow(N,\eta)$ be a symplectic embedding. Fix an admissible $D\Subset\operatorname{Int}M$ and
a compatible $J$ near $\overline D$. Put $J'=\varphi_*J$. Then $\varphi^*g_{J'}=g_J$.
Transporting the fields gives a homeomorphism of configuration
quotients. This homeomorphism preserves the energy and the evaluation
Chern class. Hence it preserves $F^*(\mu^k)\ne0$ on every finite parameter
complex. By \eqref{c:eq:naturality}, we have
\[
 W_k(D,g_J)=W_k(\varphi(D),g_{J'}).
\]
 Taking the two suprema, we obtain $\cAH_k(M,\omega)\le\cAH_k(N,\eta)$.

For $r>0$, the compatible structures for $\omega$ and $r\omega$
are the same, and $g_{J,r\omega}=rg_{J,\omega}$. By \eqref{eq:scale-family},
\[
 \cAH_k(M,r\omega)=r\cAH_k(M,\omega).
\]
Monotonicity in $k$ follows from Lemma~\ref{c:lem:elementary} and is preserved
when we take the suprema.

By Lemma~\ref{c:prop:ball-first}, we have \eqref{c:eq:raw-first}. For $0<b<a$, the ball with area
parameter $b$ is an admissible inner domain of $B(a)$. Thus
\[
 \cAH_1(B(a))\ge W_1(B(b),g_0)=b.
\]
We let $b\uparrow a$. By monotonicity and Lemma~\ref{c:prop:cylinder},
\[
 a\le\cAH_1(B(a))\le\cAH_1(Z(a))\le a.
\]
This proves the normalization of the first capacity. For every $k\ge1$, monotonicity
in the index and the cylinder families give
\begin{equation}\label{c:eq:nontrivial-all}
 a\le\cAH_k(B(a))\le\cAH_k(Z(a))\le k^2a.
\end{equation}
Therefore every fixed term satisfies the nontriviality axiom.
\end{proof}

\section{Localization with integer coefficients}
\label{c:sec:localization}
In this section, we prove a localization inequality for
AH sweepouts parametrized by finite simplicial complexes.
We first fix a smooth metric and use its restrictions
on the subdomains. We omit the metric from the notation.

The proof follows the Lusternik--Schnirelmann localization argument of
Gromov \cite[Section 8]{Gromov2003}; see also Guth
\cite[Section 3]{Guth2009}. Liokumovich--Marques--Neves use this argument
for the volume spectrum \cite[Lemma 2.15 and Section 3]{LMN2018}.
For the Allen--Cahn energy, Gaspar--Guaraco prove a corresponding
finite-scale inequality in the proof of \cite[Lemma 3.2]{GasparGuaraco2019}.
We use the same localization principle with the  evaluation Chern
class on the AH gauge quotient.

\begin{lemma}\label{c:lem:restriction}
Let $V\subset U$ be connected  domains with the same restricted
metric. Then
\begin{equation}\label{c:eq:restriction-finite}
 w_{\eps,k}(U)\ge
 w_{\eps,k}(V).
\end{equation}
Consequently, $W_k(V)\le W_k(U)$.
\end{lemma}
\begin{proof}[\bf Proof of Lemma \ref{c:lem:restriction}]
Put $b_V=\Vol(V)/(8\pi\eps^2)$.
By Lemma~\ref{c:lem:elementary}, we have
$w_{\eps,k}(V)\le b_V$.
If $w_{\eps,k}(U)\ge b_V$, the conclusion follows immediately.
Suppose that $w_{\eps,k}(U)<b_V$, and fix
$0<\delta<b_V-w_{\eps,k}(U)$.
Choose an AH $k$-sweepout $F:K\to\B^*(U)$ whose maximum
energy $\Lambda$ satisfies
\[
 \Lambda<w_{\eps,k}(U)+\delta<b_V.
\]
For every representative $(u,A)$ of a point in $F(K)$, the
restriction $u|_V$ is not identically zero. Otherwise,
\eqref{c:eq:zero-barrier} would give
\[
 b_V\le\En^U_{\eps}(u,A)\le\Lambda<b_V,
\]
which is a contradiction.

Restriction therefore defines a continuous map
$F_V:K\to\B^*(V)$.
Choose an evaluation point in $V$. The evaluation characters
on $U$ and $V$ agree at this point, so
\[
 F_V^*\mu_V=F^*\mu_U.
\]
Thus $F_V$ is an AH $k$-sweepout. Since restriction does not
increase the energy,
\[
 w_{\eps,k}(V)
 \le\max_K\En^V_{\eps}\circ F_V
 \le\Lambda<w_{\eps,k}(U)+\delta.
\]
Letting $\delta\downarrow0$ proves
\eqref{c:eq:restriction-finite}.
Taking the lower limit as $\eps\downarrow0$ gives
$W_k(V)\le W_k(U)$.
\end{proof}

We record the elementary relative cup product property  used below explicitly
(cf.\cite[Section 3.2, pp.~209--210]{Hatcher}).

\begin{lemma}\label{lem:finite-cup-cover}
Let $K$ be a finite simplicial complex covered by subcomplexes
$K_1,\ldots,K_N$. Suppose that $\alpha_i\in H^{q_i}(K;\Z)$ and $\alpha_i|_{K_i}=0$. Then
\[
 \alpha_1\smile\cdots\smile\alpha_N=0.
\]
\end{lemma}
\begin{proof}[\bf Proof of Lemma \ref{lem:finite-cup-cover}]
By the exact sequence of the pair $(K,K_i)$, we can lift $\alpha_i$ to
$H^{q_i}(K,K_i;\Z)$. We use simplicial cochains to calculate the product. The
canonical comparison with singular cohomology preserves cup products
\cite[Sections 3.1--3.2]{Hatcher}. Choose a representing cocycle $a_i$ which vanishes on
$K_i$, and order the vertices of $K$. In the cup product on a
simplex, each $a_i$ is evaluated on a face of that simplex. Every
simplex of $K$ belongs to some $K_i$, because these
subcomplexes cover $K$. The face also belongs to
$K_i$. Therefore the $i$th factor is zero. This proves that
the product cochain vanishes on every simplex. Equivalently, the
relative product lies in $H^{\sum_iq_i}(K,\bigcup_iK_i;\Z)=0$.
\end{proof}

\begin{lemma}\label{c:prop:fixed-local}
Let $U_1,\ldots,U_N$ be pairwise disjoint connected admissible subdomains of
$U$, with the restricted metric. For $k_i\ge1$ and $k=\sum_i k_i$,
we have
\begin{equation}\label{c:eq:fixed-local}
 W_k(U)\ge\sum_i W_{k_i}(U_i).
\end{equation}
More precisely, for every $\eps>0$,
\begin{equation}\label{c:eq:finite-local}
 w_{\eps,k}(U)\ge\sum_iw_{\eps,k_i}(U_i).
\end{equation}
\end{lemma}
\begin{proof}[\bf Proof of Lemma \ref{c:prop:fixed-local}]
Fix $\eps>0$, and put $a_i=w_{\eps,k_i}(U_i)$. By
Lemma~\ref{c:lem:elementary},
\[
 0\le a_i\le\frac{\Vol(U_i)}{8\pi\eps^2}<\infty.
\]
Suppose that there is an AH $k$-sweepout $F:K\to\B^*(U)$ whose
maximum energy $\Lambda$ satisfies $\Lambda<\sum_i a_i$. Choose
\[
 \Lambda<L<\sum_i a_i,
 \qquad
 L_i=a_i-\frac{\sum_j a_j-L}{N}.
\]
Then $L_i<a_i$ and $\sum_iL_i=L>\Lambda$. Some $L_i$ may be
nonpositive.

For each $z\in K$, choose any representative $(u,A)$ of $F(z)$ and set
\[
 e_i(z)=\En^{U_i}_\eps(u|_{U_i},A|_{U_i}).
\]
Here the right side is the integral of the AH energy density over $U_i$.
It is defined even when $u|_{U_i}\equiv0$. It is gauge invariant and
continuous in $z$, so no continuous choice of representatives is needed.
Since the subdomains are disjoint and the energy density is nonnegative,
\[
 \sum_i e_i(z)\le\En^U_\eps(F(z))\le\Lambda
 \qquad(z\in K).
\]
Define the open subsets
\[
 \mathscr O_i=\{z\in K:e_i(z)<L_i\}.
\]
These sets cover $K$. Otherwise, at a point outside their union we would
have $\sum_i e_i(z)\ge\sum_iL_i=L>\Lambda$, a contradiction.
If $L_i\le0$, then $\mathscr O_i$ is empty.

We now define the restriction only on $\mathscr O_i$. For $z\in\mathscr O_i$,
\[
 e_i(z)<L_i<a_i\le\frac{\Vol(U_i)}{8\pi\eps^2}.
\]
The zero-section estimate \eqref{c:eq:zero-barrier} therefore implies
$u|_{U_i}\not\equiv0$. Thus restriction gives a continuous map
\[
 F_i:\mathscr O_i\longrightarrow\B^*(U_i).
\]
Put $\lambda=F^*\mu_U\in H^2(K;\Z)$. By the naturality of the
evaluation line under restriction,
\[
 F_i^*\mu_{U_i}=\lambda|_{\mathscr O_i}.
\]
Indeed, for each $i$ we choose an evaluation point in $U_i$ and use the
same point on both sides. The class $\mu_U$ does not depend on that choice.
No restriction map on all of $K$ is required.

We subdivide $K$ so that each closed simplex lies in one member of the
open cover $\{\mathscr O_i\}$. Assign each simplex to one such member,
and let $K_i$ be the subcomplex generated by the simplices assigned to
$\mathscr O_i$. Then
\[
 K=\bigcup_iK_i,\qquad K_i\subset\mathscr O_i.
\]
Each nonempty $K_i$ is compact, and hence
$\max_{K_i}e_i<L_i<a_i$. If $\lambda^{k_i}|_{K_i}\ne0$, then
$F_i|_{K_i}$ would be an AH $k_i$-sweepout on $U_i$ with maximum
energy strictly less than $w_{\eps,k_i}(U_i)=a_i$. This contradicts
the definition of the local min--max value. Therefore
$\lambda^{k_i}|_{K_i}=0$ for every $i$, including the case
$K_i=\varnothing$.

We apply Lemma~\ref{lem:finite-cup-cover} with $\alpha_i=\lambda^{k_i}$
and obtain
\[
 F^*(\mu_U^k)
 =\lambda^{k_1}\smile\cdots\smile\lambda^{k_N}
 =0.
\]
This contradicts the sweepout condition. Thus every AH $k$-sweepout
has maximum energy at least $\sum_i a_i$, which proves
\eqref{c:eq:finite-local}.

We keep the finite collection of domains and all indices fixed and let
$\eps\downarrow0$. The finite-scale inequality gives
\[
 \begin{aligned}
 W_k(U)
 &\ge\liminf_{\eps\downarrow0}\sum_iw_{\eps,k_i}(U_i)\\
 &\ge\sum_i\liminf_{\eps\downarrow0}w_{\eps,k_i}(U_i)
 =\sum_iW_{k_i}(U_i).
 \end{aligned}
\]
This proves \eqref{c:eq:fixed-local}, including the case in which one of
the limiting values is infinite.
\end{proof}

The following inequality compares the capacity of one connected 
manifold with the capacities of finitely many disjoint connected open
subsets. It does not require a capacity to be assigned to their
disjoint union.

\begin{lemma}\label{c:prop:optimized-local}
Let $U_1,\ldots,U_N$ be pairwise disjoint connected open subsets of a connected
symplectic four-manifold $U$. Then
\[
 \cAH_{\sum_i k_i}(U)\ge\sum_i\cAH_{k_i}(U_i),\qquad k_i\ge1.
\]
\end{lemma}
\begin{proof}[\bf Proof of Lemma \ref{c:prop:optimized-local}]
Fix smooth inner domains $D_i\Subset U_i$ and compatible structures $J_i$
near their closures. We shrink the neighborhoods so that their closures
are disjoint. Since $U$ is connected, we can join the $D_i$ by
finitely many compact paths and take a connected smooth $D\Subset U$
containing this set. By extending compatible
almost complex structures, there exists a compatible $J$ near
$\overline D$ which is equal to $J_i$ near each $\overline{D_i}$.
Thus $g_J|_{D_i}=g_{J_i}$.

We apply fixed-metric localization to this one metric and obtain
\[
 \cAH_{\sum_i k_i}(U)\ge W_{\sum_i k_i}(D,g_J)
 \ge\sum_iW_{k_i}(D_i,g_{J_i}).
\]
If the local suprema are finite, we choose the data so that each term
on the right is greater than $\cAH_{k_i}(U_i)-\delta/N$, and then let $\delta\downarrow0$. If a
local supremum is infinite, we first choose an arbitrary finite lower
bound and then let it increase. 
\end{proof}

We record the exhaustion property of the definition
\eqref{eq:family-capacity}. 

\begin{lemma}\label{lem:connected-exhaustion}
Let $U\subset\R^4$ be a nonempty connected open set. For every fixed
$k\ge1$,
\[
 \cAH_k(U,\omega_0)
 =\sup_{\substack{D\Subset U,\ D\text{ connected}\\
                  \overline D\text{ compact with smooth boundary}}}
       \cAH_k(D,\omega_0).
\]
Moreover, if $D_1\Subset D_2\Subset\cdots$ is an exhaustion
of $U$ by bounded connected smooth domains, then
\begin{equation}\label{eq:open-exhaustion}
 \cAH_k(U,\omega_0)=\sup_j\cAH_k(D_j,\omega_0)
                  =\lim_{j\to\infty}\cAH_k(D_j,\omega_0).
\end{equation}
\end{lemma}

\begin{proof}[\bf Proof of Lemma \ref{lem:connected-exhaustion}]
Monotonicity gives $\cAH_k(D,\omega_0)\le\cAH_k(U,\omega_0)$
for every domain $D$ in the first supremum.
For the reverse inequality, fix a connected smooth domain
$V\Subset U$ and a compatible structure $J$ near $\overline V$.
Choose a connected smooth domain $D$ such that
$\overline V\subset D\Subset U$.
The definition of the capacity gives
\[
 W_k(V,g_J)\le\cAH_k(D,\omega_0).
\]
Taking the supremum over $V$ and $J$ proves the first equality.

For the given exhaustion, compactness of $\overline V$ and the
inclusions $D_j\subset D_{j+1}$ imply that
$\overline V\subset D_j$ for some $j$.
Hence
\[
 W_k(V,g_J)\le\cAH_k(D_j,\omega_0)
 \le\sup_j\cAH_k(D_j,\omega_0).
\]
Taking the same suprema gives
$\cAH_k(U,\omega_0)\le\sup_j\cAH_k(D_j,\omega_0)$.
The opposite inequality follows from monotonicity.
The sequence $\cAH_k(D_j,\omega_0)$ is nondecreasing, so its
limit equals its supremum, also when this value is infinite.
\end{proof}

\section{The Seiberg--Witten equations and the Abelian--Higgs energy}
\label{m:sec-comparison}
In this section, we use the AH energy \eqref{eq:unitary-AH-energy} on a possibly
nontrivial Hermitian line bundle. We follow the notation and normalization in
\cite{Taubes1996,Taubes2000}.  
We use Taubes's estimates to bound the AH energy of the scalar component.

\subsection{The perturbed Seiberg--Witten equations}
Let $(X,\omega)$ be a closed symplectic four-manifold. Fix an
$\omega$-compatible almost complex structure $J$, and put
$g=\omega(\cdot,J\cdot)$. We use the symplectic orientation. Then
$*_g\omega=\omega$, $|\omega|_g^2=2$, and $dV_g=\omega^2/2$.
 Let $K$ be the canonical line
bundle of $(X,J)$. Let $E\to X$ be a Hermitian line bundle such that
$c_1(E)\cdot[\omega]>0$. A dot denotes the intersection pairing on $X$.

We write $S_+$ and $S_-$ for the positive and negative spinor bundles.
The positive spinor bundle and its determinant line are
\[
 S_+=E\oplus(K^{-1}\otimes E),\qquad
 L=\det(S_+)=K^{-1}\otimes E^2.
\]
Following \cite[Section 1(a), equations (1.4)--(1.6)]{Taubes1996}, we denote
Clifford multiplication by $c$ and its induced maps by
\[
 \sigma:S_+\otimes T^*X\longrightarrow S_-,\qquad
 c_+:\Lambda_+\longrightarrow\operatorname{End}(S_+),\qquad
 \tau:\operatorname{End}(S_+)\longrightarrow\Lambda_+\otimes\C.
\]
Here $\Lambda_+$ is the bundle of self-dual two-forms. The map $\tau$ is
normalized as in Taubes's equation\cite[(1.5)]{Taubes1996}. In an oriented orthonormal
coframe $(e^1,\ldots,e^4)$, it satisfies
\[
 \tau(\eta\otimes\eta^*)
 =-\frac12\sum_{j,k=1}^4
   \langle\eta,c(e^j)c(e^k)\eta\rangle\,e^j\wedge e^k.
\]
We take Hermitian inner products to be conjugate-linear in the first entry.
In particular, $\tau(\eta\otimes\eta^*)$ is imaginary-valued.
We write
\[
 P_+=\frac12(1+*_g),\qquad P_-=\frac12(1-*_g)
\]
for the projections on self-dual and anti-self-dual two-forms.

Let $A_0$ be Taubes's canonical connection on $K^{-1}$. Let $u_0$ be the
unit section of the trivial summand of the canonical positive spinor
bundle $\C\oplus K^{-1}$. The projection of $\nabla_{A_0}u_0$ to that
trivial summand is zero, and
\[
 D_{A_0}u_0=0.
\]
These are the conventions in \cite[Section 1(c), equations (1.12)--(1.13)]{Taubes1996}.

Let $a$ be a unitary connection on $E$. The corresponding connection on
$L$ is denoted by $A$. Thus
\[
 A=A_0+2a,\qquad F_A=F_{A_0}+2F_a,\qquad
 F_a=\nabla_a^2\in i\Omega^2(X),\qquad
 \left[\frac{iF_a}{2\pi}\right]=c_1(E).
\]
 A gauge
transformation $\gamma:X\to S^1$ acts by
\[
 (A,\psi)\longmapsto(A-2\gamma^{-1}d\gamma,\gamma\psi),
 \qquad a\longmapsto a-\gamma^{-1}d\gamma.
\]
These are the gauge conventions in \cite[Section 1(b)]{Taubes1996}.
The connection $A$ and the Levi--Civita connection induce the spinor derivative
$\nabla_A$. We write $D_A=\sigma\circ\nabla_A$. Thus $A$ is a connection
on the determinant line, and $a$ is the connection used in the AH energy.
For general spin$^c$ connections, we also refer to
\cite[Section 2(i)--(ii)]{KronheimerMrowka1997}. We also write $D_A$ for the Dirac operator on
$S_+\oplus S_-$. It is formally self-adjoint.

If $u$ is a section of $E$, we write $u u_0$ for its image in the first
summand of $S_+$. By \cite[equation (1.15)]{Taubes1996},
\[
 \nabla_A(u u_0)=(\nabla_a u)u_0+u\nabla_{A_0}u_0,
 \qquad D_A(u u_0)=\sigma(u_0\otimes\nabla_a u).
\]
For a section $\beta$ of $K^{-1}\otimes E$, we use Taubes's projected
covariant derivative
\[
 \nabla'_A\beta
 =\frac12\left(1-\frac{i}{2}c_+(\omega)\right)\nabla_A\beta.
\]
See \cite[(1.16)]{Taubes1996}. Clifford multiplication by
$\omega$ has eigenvalues $-2i$ and $2i$ on the two summands of $S_+$.

We continue to write $\langle iF_a,\omega\rangle_g$ for the pointwise
inner product of real two-forms. In an orthonormal frame
$(e_1,Je_1,e_2,Je_2)$, we have
\[
 \langle iF_a,\omega\rangle_g
 =iF_a(e_1,Je_1)+iF_a(e_2,Je_2),\qquad
 \langle iF_a,\omega\rangle_g\,dV_g=iF_a\wedge\omega.
\]
Since $|\omega|_g^2=2$, the projection of $iF_a$ onto $\R\omega$ is
$\tfrac12\langle iF_a,\omega\rangle_g\,\omega$.

On a unitary trivializing chart, we keep the AH notation from
\eqref{eq:SW-AH-notation}:
\[
 \nabla_a=d-iA_a,\qquad F_a=-i\,dA_a,\qquad
 \En^D_{\eps,g}(u,a)=\En^D_{\eps,g}(u,A_a).
\]
Here $A_a$ is a real one-form. It is not the determinant connection $A$.
The real one-form denoted by $A$ in the earlier AH sections is denoted by
$A_a$ when it is obtained from the SW connection $a$.

We use precisely the perturbation in
\cite[Section 1(d), equations (1.17)--(1.18)]{Taubes1996}. For $r>0$, put
\[
 \mu_r=-\frac{ir}{4}\omega+P_+F_{A_0}.
\]
The perturbed SW equations for $(A,\psi)$ are
\[
 D_A\psi=0,\qquad
 P_+F_A=P_+F_{A_0}+\frac14\bigl(\tau(\psi\otimes\psi^*)-ir\omega\bigr).
\]

As in Taubes's equation \cite[(1.19)]{Taubes1996}, we write
\[
 \psi=r^{1/2}(\alpha u_0+\beta),\qquad
 \alpha\in\Gamma(E),\quad \beta\in\Gamma(K^{-1}\otimes E).
\]
By Taubes's equation (1.20), the equations are equivalent to
\begin{align}
 \sigma(u_0\otimes\nabla_a\alpha)+D_A\beta&=0,
       \label{m:SW-dirac}\\
 P_+F_a&=-\frac{ir}{8}(1-|\alpha|^2+|\beta|^2)\omega
       +\frac{ir}{4}(\alpha\beta^*+\alpha^*\beta).
       \label{m:SW-curvature}
\end{align}
The expressions $\alpha\beta^*$ and $\alpha^*\beta$ are sections of $K$
and $K^{-1}$. We regard them as the conjugate summands of
$\Lambda_+\otimes\C$, with the metric identifications used in
\cite[equation (1.20)]{Taubes1996}. They are orthogonal to $\omega$, and
\[
 |\alpha\beta^*+\alpha^*\beta|^2=2|\alpha|^2|\beta|^2.
\]
This normalization also gives the norm formula in
\cite[equation (2.12)]{Taubes1996}.
Taking the inner product of $iF_a$ with $\omega$ in
\eqref{m:SW-curvature}, we obtain
\begin{equation}\label{m:SW-moment}
 \langle iF_a,\omega\rangle_g
 =\frac r4(1-|\alpha|^2+|\beta|^2).
\end{equation}
We set
\begin{equation}\label{m:scale}
 \eps_r=\sqrt{\frac2r}.
\end{equation}
Then $1/(2\eps_r^2)=r/4$. This is the AH scale for Taubes's parameter
$r$ in \eqref{m:SW-curvature}. 

If $J$ is integrable and $\beta=0$, these are the K\"ahler vortex
equations, written using the Dirac operator and $P_+F_a$.
For the flat product model, let
$\omega=dx_1\wedge dx_2+dx_3\wedge dx_4$ and pull back the fields from
the $(x_1,x_2)$-plane. Then \eqref{m:SW-dirac}--\eqref{m:SW-curvature}
with $\beta=0$ reduce to
\[
 (\nabla_{a,1}+i\nabla_{a,2})\alpha=0,\qquad
 iF_a=\frac{r}{4}(1-|\alpha|^2)\,dx_1\wedge dx_2
     =\frac{1-|\alpha|^2}{2\eps_r^2}\,dx_1\wedge dx_2.
\]
Indeed, $P_+(dx_1\wedge dx_2)=\omega/2$. This agrees with the planar
vortex normalization in \eqref{eq:planar-vortex} and with
\cite[equation (4.4)]{Taubes1996} when $r=1$.

\subsection{Taubes's estimates}
All norms in this section are taken on $(X,g)$. We use Propositions 2.1,
2.3, and 2.8 and Lemma 2.6 of \cite{Taubes1996}. These estimates are
stated for the equations \eqref{m:SW-dirac}--\eqref{m:SW-curvature} with
the normalization above. We only integrate the stated estimates and
apply the Dirac equation.

\begin{lemma}\label{m:taubes-input}
There exist constants $C,r_0>0$, depending on $(X,\omega,J,E)$, such that
every solution of \eqref{m:SW-dirac}--\eqref{m:SW-curvature} with
$r\ge r_0$ satisfies
\begin{equation}\label{m:SW-estimate}
 \|D_A(\alpha u_0)\|_2^2
 +r\int_X|\alpha|^2|\beta|^2\,dV_g
 +r\|\beta\|_4^4\le\frac Cr.
\end{equation}
The constants are independent of the solution.
\end{lemma}
\begin{proof}[\bf Proof of Lemma \ref{m:taubes-input}]
 Propositions 2.1 and 2.3 and Lemma 2.6 of
\cite{Taubes1996} give, after increasing the constants,
\[
 \|\alpha\|_\infty\le C,\qquad
 |\beta|^2\le\frac Cr\left((1-|\alpha|^2)+\frac C{r^2}\right),\qquad
 \int_X|1-|\alpha|^2|\,dV_g\le\frac Cr.
\]
Since $1-|\alpha|^2\le1$, the second inequality gives
$\|\beta\|_\infty^2\le C/r$. Integrating it gives
$\|\beta\|_2^2\le C/r^2$. By Proposition 2.8, equation (2.37), of
\cite{Taubes1996},
\[
 |\nabla_a\alpha|^2+r|\nabla'_A\beta|^2
 \le C_1r(1-|\alpha|^2)+C_2.
\]
We integrate this inequality and use the bound for $\int_X|1-|\alpha|^2|$. Then
\[
 \|\nabla_a\alpha\|_2^2+r\|\nabla'_A\beta\|_2^2\le C.
\]
The part of $\nabla_A\beta$ perpendicular to $K^{-1}\otimes E$ is a
fixed bundle map applied to $\beta$. It is independent of $a$, since
twisting by $a$ acts by a scalar and preserves the splitting. Hence
\[
 \|D_A\beta\|_2^2
 \le C\bigl(\|\nabla'_A\beta\|_2^2+\|\beta\|_2^2\bigr)
 \le C/r.
\]
By \eqref{m:SW-dirac}, we have $D_A(\alpha u_0)=-D_A\beta$.
Also,
\[
 r\int_X|\alpha|^2|\beta|^2\,dV_g\le C/r,\qquad
 r\|\beta\|_4^4
 \le r\|\beta\|_\infty^2\|\beta\|_2^2\le C/r^2.
\]
These inequalities prove \eqref{m:SW-estimate}.
\end{proof}

\subsection{The AH energy of the scalar component}
We write the square completion using $D_A$ and the self-dual projection
$P_+$. 
The following identity is a calculation on the closed manifold $X$.

\begin{lemma}\label{m:basic-identities}
For every smooth section $u$ of $E$ and every unitary connection $a$,
with $A=A_0+2a$, we have
\begin{equation}\label{m:BPSidentity}
\begin{split}
 2\pi\En^X_{\eps,g}(u,a)
 ={}&2\pi c_1(E)\cdot[\omega]-4\pi^2\eps^2c_1(E)^2
       +\|D_A(u u_0)\|_2^2\\
    &+2\eps^2\left\|P_+F_a+
                 \frac{i(1-|u|^2)}{4\eps^2}\omega\right\|_2^2.
\end{split}
\end{equation}
Here $c_1(E)^2=\langle c_1(E)\smile c_1(E),[X]\rangle$.
\end{lemma}
\begin{proof}[\bf Proof of Lemma \ref{m:basic-identities}]
Recall that  \[
 2\pi\En^X_{\eps,g}(u,a)=\int_X
 \left(|\nabla_a u|_g^2+\eps^2|F_a|_g^2+
       \frac{(1-|u|^2)^2}{4\eps^2}\right)dV_g.
\]

We first verify the identities
\begin{align}
 \|\nabla_a u\|_2^2
 &=\|D_A(u u_0)\|_2^2
   +\int_X\langle iF_a,\omega\rangle_g|u|^2\,dV_g,
       \label{m:scalar-identity}\\
 |F_a|^2dV_g
 &=2|P_+F_a|^2dV_g-(iF_a)\wedge(iF_a).
       \label{m:form-identity}
\end{align}
For \eqref{m:scalar-identity},  we use the Weitzenb\"ock formula
in the conventions of \cite[Section 2(a)]{Taubes1996}:
\[
 D_A^2
 =\nabla_A^*\nabla_A+\frac R4
  +\frac12c_+(P_+F_A)
 \quad\text{on }S_+,
\]
where $R$ is the scalar curvature of $g$.

Fix a point and choose a local orthonormal frame
$e_1,\ldots,e_4$ whose covariant derivatives vanish at
that point. 
The Leibniz rule for the tensor product connection gives
\[
\begin{aligned}
 \nabla_A^*\nabla_A(uu_0)
 ={}&(\nabla_a^*\nabla_a u)u_0
     -2\sum_{\ell=1}^4
       (\nabla_{a,e_\ell}u)\nabla_{A_0,e_\ell}u_0\\
    &+u\nabla_{A_0}^*\nabla_{A_0}u_0.
\end{aligned}
\]
Since $F_A=F_{A_0}+2F_a$, the Weitzenb\"ock formulas for
$A$ and $A_0$ imply
\[
\begin{aligned}
 D_A^2(uu_0)
 ={}&(\nabla_a^*\nabla_a u)u_0
     -2\sum_{\ell=1}^4
       (\nabla_{a,e_\ell}u)\nabla_{A_0,e_\ell}u_0\\
    &+uD_{A_0}^2u_0+c_+(P_+F_a)(uu_0).
\end{aligned}
\]
Here the operators involving $A_0$ act on the canonical
spinor factor. By the choice of $A_0$, we have
$D_{A_0}u_0=0$ and $\langle\nabla_{A_0}u_0, u_0\rangle=0$.
Thus the term $uD_{A_0}^2u_0$ vanishes, and the second term
has zero inner product with $uu_0$.

It remains to calculate the curvature term.
The component of $P_+F_a$ along $\omega$ is
\[
 -\frac i2\langle iF_a,\omega\rangle_g\,\omega,
\]
since $|\omega|_g^2=2$. By the Clifford relations, the action of the component
perpendicular to $\omega$ maps the first summand $E$ into
$K^{-1}\otimes E$. Since
$c_+(\omega)u_0=-2iu_0$, we obtain
\[
\begin{aligned}
 \langle uu_0,c_+(P_+F_a)(uu_0)\rangle
 &=
 -\frac i2\langle iF_a,\omega\rangle_g|u|^2
       \langle u_0,c_+(\omega)u_0\rangle\\
 &=-\langle iF_a,\omega\rangle_g|u|^2.
\end{aligned}
\]

We now integrate over the closed manifold $X$.
Formal self-adjointness of the full Dirac operator and
integration by parts give
\[
\begin{aligned}
 \|D_A(uu_0)\|_2^2
 &=
 \int_X\operatorname{Re}
       \langle uu_0,D_A^2(uu_0)\rangle\,dV_g\\
 &=
 \int_X\operatorname{Re}
       \langle u,\nabla_a^*\nabla_a u\rangle\,dV_g
 -\int_X\langle iF_a,\omega\rangle_g|u|^2\,dV_g\\
 &=
 \|\nabla_a u\|_2^2
 -\int_X\langle iF_a,\omega\rangle_g|u|^2\,dV_g.
\end{aligned}
\]
This proves \eqref{m:scalar-identity}.

\eqref{m:form-identity} follows directly from the orthogonal decomposition into
self-dual and anti-self-dual parts of the real two-form $iF_a$.

Since $|\omega|_g^2=2$, we have
\[
 2\eps^2\left|P_+F_a+\frac{i(1-|u|^2)}{4\eps^2}\omega\right|^2
 =2\eps^2|P_+F_a|^2-(1-|u|^2)\langle iF_a,\omega\rangle_g+\frac{(1-|u|^2)^2}{4\eps^2}.
\]

The remaining integrals are
\[
 \int_X iF_a\wedge\omega=2\pi c_1(E)\cdot[\omega],\qquad
 \int_X(iF_a)\wedge(iF_a)=4\pi^2c_1(E)^2.
\]
This proves \eqref{m:BPSidentity}.
\end{proof}

\begin{lemma}\label{m:SW-AH}
For solutions of \eqref{m:SW-dirac}--\eqref{m:SW-curvature}, we have
\begin{equation}\label{m:energyasymptotic}
 \En^X_{\eps_r,g}(\alpha,a)=c_1(E)\cdot[\omega]+o(1)
 \qquad(r\to\infty).
\end{equation}
The error is uniform over all solutions for the fixed data $(X,\omega,J,E)$.
\end{lemma}
\begin{proof}[\bf Proof of Lemma \ref{m:SW-AH}]
By \eqref{m:SW-curvature} and $\eps_r^2=2/r$, we have
\[
 P_+F_a+\frac{i(1-|\alpha|^2)}{4\eps_r^2}\omega
 =-\frac{ir}{8}|\beta|^2\omega
      +\frac{ir}{4}(\alpha\beta^*+\alpha^*\beta).
\]
The two terms on the right are orthogonal. Hence
\[
 2\eps_r^2\left\|P_+F_a+
          \frac{i(1-|\alpha|^2)}{4\eps_r^2}\omega\right\|_2^2
 =\frac r2\int_X|\alpha|^2|\beta|^2\,dV_g
      +\frac r8\|\beta\|_4^4=O(r^{-1}).
\]
By Lemma~\ref{m:taubes-input}, we also have
$\|D_A(\alpha u_0)\|_2^2=O(r^{-1})$.
The topological correction in \eqref{m:BPSidentity} is
$-8\pi^2c_1(E)^2/r$. We substitute these estimates into
\eqref{m:BPSidentity} and divide by $2\pi$. This proves
\eqref{m:energyasymptotic}.
\end{proof}

In the next section, we restrict $(\alpha,a)$ to a domain on
which $E$ is trivial. Since the energy density is nonnegative,
the estimate obtained above on the closed manifold $X$ also
bounds the energy of the restricted pairs.

\section{Sweepouts from Seiberg--Witten moduli spaces}
\label{m:sec-upper}
The purpose of this section is to prove
Theorem~\ref{m:thm:models}.
We use regular SW moduli spaces on $\CP^2$ and $S^2\times S^2$.
We first calculate the SW invariants in the symplectic chamber
and construct AH sweepouts.
We then prove an upper bound for the capacities of the ball
and compute the capacities of the cylinder.

The use of the SW index and the symplectic area is related to
\cite[Proposition 2.16]{ChaidezWormleighton2021},
\cite[Lemma~18]{Hutchings2022}, and
\cite[Lemma 3.18]{Beiner2026}. These works use the index and area
to bound ECH or alternative capacities. Here we use the SW moduli space and its evaluation class
with integer coefficients to construct an AH sweepout.

\subsection{The SW moduli space and wall crossing}
We use the construction of the moduli space in
\cite[Section 1(b)]{Taubes1996}. For compactness, orientations, and
transversality, we refer to \cite{Morgan1996}.
Let $(X,\omega)$ be a closed symplectic four-manifold with
$b_1(X)=0$ and $b_2^+(X)=1$. Let $E\to X$ be a Hermitian line
bundle, and put $e=c_1(E)$. We assume
\[
 e\cdot[\omega]>0,\qquad
 c_1(L)\cdot[\omega]>0,\qquad
 2m=e^2-c_1(K)\cdot e,\quad m\in\Z_{\ge1},
\]
where $L=K^{-1}\otimes E^2$ is the determinant line. We also assume
that $\omega$ has a compatible metric of positive scalar curvature.
We will check these assumptions for the two models below.
For an arbitrary compatible structure $J$, put
$g=\omega(\cdot,J\cdot)$ and use the equations of
Section~\ref{m:sec-comparison}. The virtual real dimension is
\begin{equation}\label{m:general-dimension}
 \frac{c_1(L)^2-(2\chi(X)+3\operatorname{sign}(X))}{4}
 =e^2-c_1(K)\cdot e=2m.
\end{equation}

We use the homology orientation for which $[\omega]$ is positive.
Let $\mathcal M$ be a regular moduli space with no reducible
solutions, and fix $x_0\in X$. First divide the solutions by
\[
 \G_{x_0}(X)=\{\gamma\in C^\infty(X,S^1):\gamma(x_0)=1\}.
\]
The constant gauge transformations give a principal circle bundle
$\mathcal M^0\to\mathcal M$. Its associated line is denoted by
$\LL^{\mathrm{SW}}_{x_0}$, and we put
\[
 U=c_1(\LL^{\mathrm{SW}}_{x_0})\in H^2(\mathcal M;\Z).
\]
We use the same quotient convention as in \eqref{c:eq:eval-line}.
For regular data in a fixed chamber, the SW invariant is the
integer $\langle U^m,[\mathcal M]\rangle$.

We now calculate this invariant. We use the integer wall-crossing
formula in \cite[Theorem 2.2]{ChaidezWormleighton2021} and the
vanishing argument in the proof of Proposition 2.5 of that paper.
A solution is reducible when its spinor is identically zero.
For the equation
\[
 D_A\psi=0,\qquad
 P_+F_A=\frac14\tau(\psi\otimes\psi^*)+\mu,
 \qquad \mu\in i\Omega_g^+(X),
\]
a reducible solution must satisfy $P_+F_A=\mu$.
Since $b_2^+(X)=1$, the space of self-dual harmonic two-forms for
$g$ is $\R\omega$. By Hodge theory, this curvature equation is
solvable exactly when
\[
\mathfrak w_g(\mu):=\int_X i\mu\wedge\omega
                 -2\pi c_1(L)\cdot[\omega]=0.
\]
This equality defines the wall. For
$\mu_r=-ir\omega/4+P_+F_{A_0}$, we obtain
\begin{equation}\label{m:general-wall}
 \begin{aligned}
 \mathfrak w_g(\mu_r)
 &=\frac r4\int_X\omega\wedge\omega
   +2\pi c_1(K^{-1})\cdot[\omega]
   -2\pi c_1(L)\cdot[\omega]\\
 &=\frac r4\int_X\omega\wedge\omega-4\pi e\cdot[\omega].
 \end{aligned}
\end{equation}
Thus sufficiently large $r$ puts the equation on the positive
side of the wall. This condition is independent of $J$.

Choose a compatible metric $g_*$ of positive scalar curvature.
For $g_*$ and $\mu=0$, the SW equations have no solutions.
Indeed, positive scalar curvature implies $\psi=0$. The remaining
equation $P_+F_A=0$ is impossible because
$c_1(L)\cdot[\omega]>0$. Hence the invariant is zero on the
negative side of the wall. The integer wall-crossing formula gives
\begin{equation}\label{m:SW-pairing}
 \langle U^m,[\mathcal M]\rangle=\pm1
\end{equation}
for every regular moduli space on the positive side. We call this
the symplectic chamber. In particular, $U^m\ne0$.
This is an integer equality, not only nonvanishing modulo two.

For the standard complex structures on $\CP^2$ and $S^2\times S^2$,
Riemann--Roch gives $m=\chi(E)-\chi(\mathcal O_X)$. This is the
index condition used for algebraic capacities \cite{Wormleighton2022}.

\subsection{Perturbation and restriction}
Keep $(X,\omega,J,E)$ fixed as above. By Lemma~\ref{m:SW-AH},
we can choose $\eta(r)\ge0$ with $\eta(r)\to0$ such that every
solution of the original equation satisfies
\[
 \En^X_{\eps_r,g}(\alpha,a)\le e\cdot[\omega]+\eta(r),
 \qquad \eps_r=\sqrt{2/r}.
\]
The original equation need not be regular. We choose nearby
regular data while keeping the original metric in the AH energy.

\begin{lemma}\label{m:regularization}
Fix a sufficiently large $r$. For every $\delta>0$, there are
regular data arbitrarily close to the original data and in the
same chamber such that their compact oriented moduli space
$\mathcal M_{r,\delta}$ has dimension $2m$ and satisfies
\begin{align}
 \langle U^m,[\mathcal M_{r,\delta}]\rangle&=\pm1,
       \label{m:top-pair}\\
 \sup_{\mathcal M_{r,\delta}}\En^X_{\eps_r,g}(\alpha,a)
       &\le e\cdot[\omega]+\eta(r)+\delta.
       \label{m:nearby-bound}
\end{align}
\end{lemma}
\begin{proof}[\bf Proof of Lemma~\ref{m:regularization}]
We fix $r$ and $\delta>0$. For metrics near $g$, we identify the
spinor bundles and their determinant lines with the original
bundles. These identifications depend smoothly on the metric and
commute with gauge transformations. We use the original splitting
and the fixed connection $A_0$ to write
\[
 \psi=r^{1/2}(\alpha u_0+\beta),\qquad A=A_0+2a.
\]
Thus $(\alpha,a)$ is a configuration on the fixed line bundle $E$.
We compute its AH energy using the original metric $g$.

We first show that \eqref{m:nearby-bound} holds for all solutions
when the metric and perturbation are sufficiently close to the
original data. Suppose otherwise. Then there are metrics and
perturbations converging to the original data, with solutions
$(A_j,\psi_j)$ such that
\[
 \En^X_{\eps_r,g}(\alpha_j,a_j)
 >e\cdot[\omega]+\eta(r)+\delta.
\]
By SW compactness, after applying gauge transformations and passing
to a subsequence, these solutions converge smoothly to a solution
$(A_\infty,\psi_\infty)$ of the original equation. Continuity of
the energy gives
\[
 e\cdot[\omega]+\eta(r)+\delta
 \le\En^X_{\eps_r,g}(\alpha_\infty,a_\infty)
 \le e\cdot[\omega]+\eta(r),
\]
a contradiction. We now choose regular data in this neighborhood
and in the same chamber. Their moduli space is compact and
oriented, with real dimension $2m$. The invariant is unchanged
within the chamber, so \eqref{m:SW-pairing} gives
\eqref{m:top-pair}.
\end{proof}

\begin{lemma}\label{m:SW-class-match}
Let $D\Subset X$ be a connected smooth domain such that $E$ is
trivial on a neighborhood of $\overline D$. There are sequences
$r_j\to\infty$ and $\eps_j=\sqrt{2/r_j}$, compact regular oriented
SW moduli spaces $\mathcal M_j$, and continuous maps
\[
 F_j:\mathcal M_j\longrightarrow\B^*(D),\qquad
 [A,\psi]\longmapsto[\alpha|_D,A_a|_D],
\]
where $\psi=r_j^{1/2}(\alpha u_0+\beta)$ and $A=A_0+2a$ after
the background identifications. In a fixed unitary trivialization,
$\nabla_a=d-iA_a$. These maps satisfy
\begin{align}
 F_j^*\mu_D&=U_{j},
 &\langle U_{j}^m,[\mathcal M_j]\rangle&=\pm1,
       \label{m:class-identification}\\
 \max_{\mathcal M_j}\En^D_{\eps_j,g|_D}\circ F_j
       &\le e\cdot[\omega]+o(1).&&\label{m:uniform-AH-bound}
\end{align}
Consequently,
\begin{equation}\label{m:fixed-J-bound}
 W_m(D,g|_D)\le e\cdot[\omega].
\end{equation}
\end{lemma}
\begin{proof}[\bf Proof of Lemma~\ref{m:SW-class-match}]
Choose $r_j\to\infty$ and apply Lemma~\ref{m:regularization} with
$\delta_j\downarrow0$. The  moduli spaces satisfy
$\langle U_j^m,[\mathcal M_j]\rangle=\pm1$, and the scalar AH
energy is at most $e\cdot[\omega]+o(1)$ in the fixed metric $g$.
Restriction to $D$ does not increase the energy.

Before defining $F_j$, we check that $\alpha|_D\not\equiv0$ for
all solutions when $j$ is sufficiently large. Otherwise the
potential energy on $D$ alone would be at least
\[
 \frac{\Vol_g(D)}{8\pi\eps_j^2}
 =\frac{r_j\Vol_g(D)}{16\pi}\longrightarrow\infty,
\]
contrary to the energy bound.

The restriction respects gauge equivalence. A gauge transformation
$\gamma$ sends $\alpha$ to $\gamma\alpha$ and $a$ to
$a-\gamma^{-1}d\gamma$. Hence
\[
 A_a\longmapsto A_a-i\gamma^{-1}d\gamma,
\]
which is the action in \eqref{c:eq:quotient}. Thus $F_j$ is well
defined and continuous. Fix $x_0\in D$. The character defining
the SW evaluation line is $\gamma\mapsto\gamma(x_0)$.
After restriction, the same character defines the pullback of the
AH evaluation line. Therefore
\[
 F_j^*\LL_{x_0}\cong\LL^{\mathrm{SW}}_{x_0},\qquad
 F_j^*\mu_D=U_j\quad\text{in }H^2(\mathcal M_j;\Z).
\]
This proves \eqref{m:class-identification} and
\eqref{m:uniform-AH-bound}.

Choose a finite triangulation $\tau_j:K_j\to\mathcal M_j$ and
put $\widetilde F_j=F_j\circ\tau_j$. Since the top pairing is
$\pm1$, we have
\[
 \widetilde F_j^*(\mu_D^m)=\tau_j^*(U_j^m)\ne0.
\]
Thus $(K_j,\widetilde F_j)$ is an AH $m$-sweepout. It follows that
\[
 W_m(D,g|_D)
 \le\liminf_{j\to\infty}w_{\eps_j,m}(D,g|_D)
 \le e\cdot[\omega].
\]
This proves \eqref{m:fixed-J-bound}. We keep $X,\omega,J,E,D$
fixed while taking this limit.
\end{proof}

\subsection{Estimates on balls and cylinders}
\paragraph*{The ball.}
Write $X=\CP^2(b)$ for the projective plane with the standard
symplectic form $\omega$ for which a projective line has area $b$.
Let $H$ be the hyperplane class, normalized by
$\langle H^2,[X]\rangle=1$. Fix $d\ge1$ and $E=\mathcal O(d)$.
Then
\[
 [\omega]=bH,\qquad c_1(E)=dH,\qquad c_1(L)=(2d+3)H.
\]
Since $\chi(X)=3$ and $\operatorname{sign}(X)=1$, we have
\begin{equation}\label{m:dimension}
 \frac{c_1(L)^2-(2\chi(X)+3\operatorname{sign}(X))}{4}
 =\frac{(2d+3)^2-9}{4}=2N_d,
 \qquad N_d=\frac{d(d+3)}2.
\end{equation}
Also, $e\cdot[\omega]=db$ and $c_1(L)\cdot[\omega]=(2d+3)b>0$.
The Fubini--Study metric has positive scalar curvature.
Thus the assumptions above hold, with $m=N_d$; see also
\cite[Example 2.7]{ChaidezWormleighton2021}.
Equation~\eqref{m:general-wall} becomes
\[
 \mathfrak w_g(\mu_r)=\frac{rb^2}{4}-4\pi db,
\]
so $r>16\pi d/b$ puts the equation in the symplectic chamber.

\begin{lemma}\label{m:sw-ball}\label{m:ball-bound}
For every $c>0$ and $k\ge1$, we have
\[
 \cAH_k(B(c),\omega_0)\le c\,d(k),\qquad
 d(k)=\min\{d\in\Z_{\ge1}:d(d+3)/2\ge k\}.
\]
\end{lemma}
\begin{proof}[\bf Proof of Lemma~\ref{m:sw-ball}]
Fix a connected smooth domain $D\Subset B(c)$, a compatible
structure $J$ near $\overline D$, an integer $d\ge1$, and $b>c$.
The standard symplectic ball chart in $\CP^2(b)$ contains $B(c)$.
Extend $J$ to a compatible structure $\widetilde J$ on $\CP^2(b)$
which agrees with $J$ near $\overline D$. Then
$g_{\widetilde J}|_D=g_J$. The line bundle $\mathcal O(d)$ is
trivial on the ball chart. Lemma~\ref{m:SW-class-match} gives
\[
 W_{N_d}(D,g_J)\le db.
\]
For $k\le N_d$, monotonicity in the index gives the same upper
bound for $W_k(D,g_J)$. Since $D$ and $J$ were arbitrary and $db$
is independent of them, \eqref{eq:family-capacity} gives
$\cAH_k(B(c),\omega_0)\le db$. We let $b\downarrow c$ and choose
$d=d(k)$.
\end{proof}

\paragraph*{The cylinder.}
Let
\[
 X=S^2(a_0)\times S^2(Na_0),\qquad
 \Omega=\operatorname{pr}_1^*\omega_1+
        \operatorname{pr}_2^*\omega_2,
\]
where the factors have areas $a_0$ and $Na_0$, and $N\ge1$ is an
integer. Put
\[
 F=[S^2\times\{\mathrm{pt}\}],\qquad
 G=[\{\mathrm{pt}\}\times S^2].
\]
Then $F^2=G^2=0$ and $F\cdot G=1$. Choose $E$ with
$e=c_1(E)=k\PD(F)$. The canonical class is
$c_1(K)=-2\PD(F)-2\PD(G)$, so
\begin{equation}\label{m:cylinder-intersections}
 e^2=0,\qquad c_1(K)\cdot e=-2k,\qquad
 e\cdot[\Omega]=ka_0.
\end{equation}
Thus the virtual real dimension is $2k$. Moreover,
\[
 c_1(L)\cdot[\Omega]=2(k+N+1)a_0>0,\qquad
 \int_X\Omega\wedge\Omega=2Na_0^2.
\]
The product of round metrics is compatible and has positive scalar
curvature. Hence the assumptions above hold with $m=k$; see also
\cite[Example 2.8]{ChaidezWormleighton2021}.
By \eqref{m:general-wall},
\[
 \mathfrak w_g(\mu_r)=\frac{rNa_0^2}{2}-4\pi ka_0.
\]
In particular, $r>8\pi k/(Na_0)$ puts the equation in the
symplectic chamber, where $\langle U^k,[\mathcal M]\rangle=\pm1$.

\begin{lemma}\label{m:sw-cylinder}
For every $a>0$ and $k\ge1$, we have
\begin{equation}\label{m:cylinder-exact}
 \cAH_k(Z^4(a),\omega_0)=ka.
\end{equation}
\end{lemma}
\begin{proof}[\bf Proof of Lemma~\ref{m:sw-cylinder}]
\emph{Upper bound.}
Fix a connected smooth domain $D\Subset Z^4(a)$, a compatible
structure $J$ near $\overline D$, and $a_0>a$. Use the compactification in
Section~\ref{sec:cylinder-proof}. For sufficiently large $N$, a
product of symplectic disk charts in $X$ contains a neighborhood
of $\overline D$. Extend $J$ to an $\Omega$-compatible structure
$\widetilde J$ on $X$ which agrees with $J$ near $\overline D$.
Choose $E$ as in \eqref{m:cylinder-intersections}. It is trivial
on the product chart. Lemma~\ref{m:SW-class-match} gives
\[
 W_k(D,g_J)\le ka_0.
\]
We let $a_0\downarrow a$ with $D,J,k$ fixed. The bound
$W_k(D,g_J)\le ka$ is independent of $D$ and $J$. Taking the
suprema in \eqref{eq:family-capacity}, we obtain
$\cAH_k(Z^4(a),\omega_0)\le ka$.

\emph{Lower bound.}
The cylinder contains $k$ disjoint copies of $B^4(a)$, obtained
by translations in the second complex coordinate. Apply
Lemma~\ref{c:prop:optimized-local} to these balls. By the normalization of the first capacity in Theorem~\ref{c:thm:main},
\[
 \cAH_k(Z^4(a),\omega_0)
 \ge\sum_{i=1}^k\cAH_1(B^4(a),\omega_0)=ka.
\]
This proves the opposite bound.
\end{proof}

\begin{proof}[\bf Proof of Theorem~\ref{m:thm:models}]
The upper bound on the ball and the cylinder formula follow from
Lemmas~\ref{m:sw-ball} and~\ref{m:sw-cylinder}. The lower bound on
the ball follows from the normalization of the first capacity and monotonicity in
$k$ in Theorem~\ref{c:thm:main}.
\end{proof}

\section{Weyl laws for domains in \texorpdfstring{$\R^4$}{R4}}
\label{m:sec-weyl}
In this section, we prove Theorems~\ref{thm:intro-weyl}
and~\ref{thm:intro-fixed-weyl}. We prove the cube and smooth-domain
limits for both sequences by the same argument. We follow the
subdivision method of Liokumovich--Marques--Neves in
\cite[Sections 1.4 and 3]{LMN2018}.
\subsection{The limit on a cube}
We write
\[
 w_k(U)=\cAH_k(U,\omega_0),\qquad
 v_k(D)=W_k(D,g_0),\qquad I=(0,1)^4.
\]
The first sequence is defined for every nonempty connected open
set $U\subset\R^4$. For the second sequence, we use bounded
connected smooth domains and open Euclidean cubes, as in
Section~\ref{sec:four-dimensional-definitions}. Until we pass to
general open sets, all domains below are of these two types. We
write $s_k$ for either $w_k$ or $v_k$ when the same argument applies.

Both sequences are nondecreasing in the index and under inclusion.
For $w_k$, this follows from Theorem~\ref{c:thm:main}. For $v_k$,
it follows from Lemmas~\ref{c:lem:elementary}
and~\ref{c:lem:restriction}. The pullback and scaling identities
in Section~\ref{c:sec:topology}, together with the capacity axioms,
show that both sequences are invariant under translations and satisfy
\begin{equation}\label{m:abstract-scale}
 s_k(rD)=r^2s_k(D),\qquad r>0.
\end{equation}

By Lemmas~\ref{c:prop:fixed-local}
and~\ref{c:prop:optimized-local}, both sequences satisfy
\begin{equation}\label{m:abstract-LS}
 s_{\sum_i k_i}(D)\ge\sum_i s_{k_i}(D_i),\qquad k_i\ge1,
\end{equation}
where $D_1,\ldots,D_m$ are pairwise disjoint connected subdomains
of $D$ of the types specified above. Only the values on the
connected ambient domain and the individual subdomains occur in
this inequality.

Choose a translate of $B(a_-)$ contained in $I$.
By translation invariance and monotonicity,
Theorem~\ref{c:thm:main} gives $w_1(I)\ge a_-$.
The same properties and Lemma~\ref{c:prop:ball-first}
give $v_1(I)\ge a_-$.
For the upper bound, choose a bounded connected smooth domain
$D_0$ and a ball $B(a_+)$ such that
$\overline I\subset D_0\Subset B(a_+)$. Restriction and
\eqref{eq:family-capacity} give
\[
 v_k(I)\le v_k(D_0)\le w_k(B(a_+)).
\]
Also, $w_k(I)\le w_k(B(a_+))$. By Lemma~\ref{m:sw-ball}, we obtain
\begin{equation}\label{m:abstract-bound}
 0<s_1(I),\qquad s_k(I)\le a_+d(k)\le C\sqrt{k}.
\end{equation}
Therefore the constants
\begin{equation}\label{x:optimized-cube-coefficient}
 b_{\AH}=\sup_{p\ge1}\frac{w_p(I)}{\sqrt p},\qquad
 b_{\AH}^{\mathrm{flat}}=\sup_{p\ge1}\frac{v_p(I)}{\sqrt p}
\end{equation}
are positive and finite. In the common arguments below, we use
\[
 (s_k,b)=(w_k,b_{\AH})\quad\text{or}\quad
 (s_k,b)=(v_k,b_{\AH}^{\mathrm{flat}}).
\]

\begin{lemma}\label{lem:cube-limit}
For either pair $(s_k,b)$ and every open Euclidean cube $Q$ whose
sides are parallel to the coordinate axes, we have
\[
 \lim_{k\to\infty}\frac{s_k(Q)}{\sqrt{k}}=b\sqrt{|Q|}.
\]
\end{lemma}
\begin{proof}[\bf Proof of Lemma \ref{lem:cube-limit}]
We subdivide $I$ into $\ell^4$ disjoint open cubes of side
$1/\ell$. By translation invariance,
\eqref{m:abstract-scale}, and \eqref{m:abstract-LS}, we have
\begin{equation}\label{m:cube-selfsimilar}
 s_{p\ell^4}(I)\ge\ell^4\ell^{-2}s_p(I)=\ell^2s_p(I).
\end{equation}
Fix $p$ and set $\ell=\lfloor(k/p)^{1/4}\rfloor$. Then
$p\ell^4\le k$ and $p\ell^4/k\to1$. By monotonicity in the index,
\[
 \liminf_{k\to\infty}\frac{s_k(I)}{\sqrt{k}}
 \ge\frac{s_p(I)}{\sqrt p}.
\]
Taking the supremum over $p$ gives the lower bound $b$. Every ratio
$s_k(I)/\sqrt{k}$ is at most $b$ by definition. Hence the limit
exists and is equal to $b$. Translation invariance and scaling give
the formula on $Q$.
\end{proof}

\subsection{From cubes to domains}
\begin{lemma}\label{lem:domain-limit}
For either pair $(s_k,b)$, let $\Omega\subset\R^4$ be a nonempty
bounded connected smooth domain or an open Euclidean cube. Then
\[
 \lim_{k\to\infty}\frac{s_k(\Omega)}{\sqrt{k}}
 =b\sqrt{|\Omega|}.
\]
\end{lemma}
\begin{proof}[\bf Proof of Lemma \ref{lem:domain-limit}]
\emph{Lower bound.}
Fix finitely many disjoint inner cubes $I_1,\ldots,I_m$ whose sides
are parallel to the coordinate axes. Put $v_i=|I_i|$ and
$S=\sum_i v_i$, and choose
\[
 k_i=\left\lfloor\frac{kv_i}{S}\right\rfloor.
\]
For sufficiently large $k$, these indices are positive and their
sum is at most $k$. By localization, monotonicity in the index,
and Lemma~\ref{lem:cube-limit}, we have
\[
 \liminf_{k\to\infty}\frac{s_k(\Omega)}{\sqrt{k}}
 \ge\sum_i b\sqrt{v_i}\sqrt{v_i/S}=b\sqrt S.
\]
Finite unions of inner cubes approximate the volume of $\Omega$.
Letting $S\uparrow|\Omega|$, we obtain the lower bound.

\emph{Upper bound.}
Choose an open cube $I_0$ containing $\overline\Omega$, and put
$V_0=|I_0|$ and $v=|\Omega|$. Fix finitely many disjoint cubes
$I_1,\ldots,I_m$ in $I_0\setminus\overline\Omega$. All these cubes
have sides parallel to the coordinate axes. Put $v_i=|I_i|$ and
$S=\sum_i v_i$. We choose
\[
 k_i=\left\lfloor\frac{kv_i}{v}\right\rfloor,
 \qquad K=k+\sum_i k_i.
\]
Localization in the outer cube gives
\[
 s_K(I_0)\ge s_k(\Omega)+\sum_i s_{k_i}(I_i).
\]
We move the sum to the left and divide by $\sqrt{k}$. Since
$K/k\to1+S/v$, Lemma~\ref{lem:cube-limit} gives
\begin{equation}\label{m:complement}
 \limsup_{k\to\infty}\frac{s_k(\Omega)}{\sqrt{k}}
 \le b\left[\sqrt{V_0(1+S/v)}-\frac S{\sqrt v}\right].
\end{equation}
The boundary of $\Omega$ has zero four-dimensional measure. Thus
we can let $S\uparrow V_0-v$. The expression in brackets converges
to $\sqrt v$, which proves the upper bound. 
\end{proof}

The case $(s_k,b)=(v_k,b_{\AH}^{\mathrm{flat}})$ gives the limits
in Theorem~\ref{thm:intro-fixed-weyl}. We prove the bounds on its
coefficient in the last subsection.

We now use only $w_k$, which is defined on every nonempty connected
open set by \eqref{eq:family-capacity}.

\begin{lemma}\label{lem:open-weyl}
Let $U\subset\R^4$ be a nonempty connected open set. Then
\[
 \liminf_{k\to\infty}\frac{w_k(U)}{\sqrt{k}}
 \ge b_{\AH}\sqrt{|U|}.
\]
If $U$ is bounded, then
\[
 \limsup_{k\to\infty}\frac{w_k(U)}{\sqrt{k}}
 \le b_{\AH}\sqrt{|\overline U|}.
\]
Consequently, the limit is $b_{\AH}\sqrt{|U|}$ when $U$ is bounded
and $|\partial U|=0$. If $|U|=+\infty$, the limit is $+\infty$.
\end{lemma}
\begin{proof}[\bf Proof of Lemma \ref{lem:open-weyl}]
For every bounded connected smooth domain $\Omega\Subset U$,
monotonicity and Lemma~\ref{lem:domain-limit} give
\[
 \liminf_{k\to\infty}\frac{w_k(U)}{\sqrt{k}}
 \ge b_{\AH}\sqrt{|\Omega|}.
\]
We take a smooth connected exhaustion of $U$, as in
Lemma \ref{lem:connected-exhaustion}. Its volumes tend to $|U|$.
This proves the lower bound, including the case of infinite volume.
The domain is fixed before taking the high-index limit.

Now suppose that $U$ is bounded. For every $\delta>0$, there is a
bounded connected smooth domain $\Omega_\delta$ such that
\[
 \overline U\subset\Omega_\delta,\qquad
 |\Omega_\delta|\le|\overline U|+\delta.
\]
Indeed, first take an open neighborhood of the compact set
$\overline U$ with volume this close. Finitely many boxes and
smoothing give a smooth neighborhood with finitely many components.
We join the components by thin tubes and smooth the junctions.
The tubes can be chosen with arbitrarily small total volume, and
we keep a positive margin around $\overline U$.

By monotonicity and Lemma~\ref{lem:domain-limit},
\[
 \limsup_{k\to\infty}\frac{w_k(U)}{\sqrt{k}}
 \le b_{\AH}\sqrt{|\overline U|+\delta}.
\]
Let $\delta\downarrow0$. If $|\partial U|=0$, then
$|\overline U|=|U|$, and the two bounds agree.
\end{proof}

\subsection{Bounds for the Weyl constants}
\begin{proof}[\bf Proof of Theorems~\ref{thm:intro-weyl}
and~\ref{thm:intro-fixed-weyl}]
Lemmas~\ref{lem:cube-limit} and~\ref{lem:domain-limit} give the
limits on cubes and smooth domains for both sequences.
Lemma~\ref{lem:open-weyl} gives the remaining bounds for the
capacities. It remains to compare and estimate the coefficients.

First, let $I_r$ be the concentric cube of side $r<1$ inside $I$.
Choose a connected smooth domain $D_r$ such that
$\overline{I_r}\subset D_r\Subset I$. Restriction and
\eqref{eq:family-capacity} give
\[
 r^2v_k(I)=v_k(I_r)\le v_k(D_r)\le w_k(I).
\]
For each fixed $k$, let $r\uparrow1$. We obtain $v_k(I)\le w_k(I)$.
Taking suprema in \eqref{x:optimized-cube-coefficient} gives
\[
 0<b_{\AH}^{\mathrm{flat}}\le b_{\AH}.
\]

Next, we prove the upper bound on $b_{\AH}$. Since
$|B(a)|=a^2/2$ and $d(k)/\sqrt{k}\to\sqrt2$,
Lemma~\ref{m:sw-ball} and the limit on the ball give
\[
 b_{\AH}\frac a{\sqrt2}
 =\lim_{k\to\infty}\frac{w_k(B(a))}{\sqrt{k}}
 \le a\sqrt2.
\]
Therefore $b_{\AH}\le2$.

For the lower bound, we use equal-ball packings. For every positive
integer $m$ and every $0<\zeta<1$, there is a symplectic embedding
\begin{equation}\label{s:equalpacking}
 \bigsqcup_{i=1}^{m^2}B\left(\frac{(1-\zeta)a}{m}\right)
 \hookrightarrow B(a).
\end{equation}
McDuff--Polterovich proved full packings for square numbers of
balls \cite{McDuffPolterovich1994}. 
We apply Lemma~\ref{c:prop:optimized-local} to the images of the
balls. Symplectic invariance and first normalization give
\[
 w_{m^2}(B(a))\ge
 m^2\frac{(1-\zeta)a}{m}=(1-\zeta)ma.
\]
First let $\zeta\downarrow0$ with $m$ fixed. Then use the limit on
the ball along $k=m^2$ to obtain
\[
 b_{\AH}\frac a{\sqrt2}
 =\lim_{m\to\infty}\frac{w_{m^2}(B(a))}{m}\ge a.
\]
Thus $b_{\AH}\ge\sqrt2$. This proves both theorems.
\end{proof}

\pdfbookmark[1]{References}{references}

\end{document}